\documentclass[a4paper, 12pt]{article}

\usepackage[sort&compress]{natbib}
\bibpunct{(}{)}{;}{a}{}{,} 

\usepackage{amsthm, amsmath, amssymb, mathrsfs, multirow, url, subfigure}
\usepackage{graphicx} 
\usepackage{ifthen} 
\usepackage{amsfonts}
\usepackage[usenames]{color}
\usepackage{fullpage}
\usepackage[shortlabels]{enumitem}

\theoremstyle{plain} 
\newtheorem{thm}{Theorem}
\newtheorem{cor}{Corollary}

\theoremstyle{definition}

\theoremstyle{remark}

\newtheorem*{bayespar}{\bf Bayesian}
\newtheorem*{freqpar}{\bf Frequentist}

\newcommand{\prob}{\mathsf{P}}

\newcommand{\pval}{\mathsf{pval}}
\newcommand{\sev}{\mathsf{sev}}

\newcommand{\nm}{{\sf N}}

\newcommand{\RR}{\mathbb{R}}

\newcommand{\YY}{\mathbb{Y}}

\newcommand{\TT}{\mathbb{T}}

\newcommand{\prior}{\mathsf{Q}}

\newcommand{\lPi}{\underline{\Pi}}
\newcommand{\uPi}{\overline{\Pi}}

\title{Possibility-theoretic statistical inference offers performance and probativeness assurances\footnote{This is an extended version of the conference paper \citep{improbe.2022}}}
\author{Leonardo Cella\footnote{Department of Statistical Sciences, Wake Forest University, {\tt cellal@wfu.edu}} \; and \; Ryan Martin\footnote{Department of Statistics, North Carolina State University, {\tt rgmarti3@ncsu.edu}}}
\date{\today}

\begin{document}

\maketitle 
\begin{abstract}
Statisticians are largely focused on developing methods that {\em perform} well in a frequentist sense---even the Bayesians. But the widely-publicized replication crisis suggests that these performance guarantees alone are not enough to instill confidence in scientific discoveries. In addition to reliably detecting hypotheses that are (in)compatible with data, investigators require methods that can {\em probe} for hypotheses that are actually supported by the data.  In this paper, we demonstrate that valid inferential models (IMs) achieve both performance and probativeness properties and we offer a powerful new result that ensures the IM's probing is reliable.  We also compare and contrast the IM's dual performance and probativeness abilities with that of Deborah Mayo's severe testing framework. 

\smallskip

\end{abstract}

\section{Introduction}
\label{S:intro}

Important decisions affecting our everyday experiences are becoming increasingly data-driven.  But is data helping us make better decisions?  In many ways, the answer is obviously yes; but in other ways the answer is less clear.  The widely-publicized replication crisis in science is one issue that raises serious concerns, so much so that, in 2019, the American Statistical Association's president commissioned a formal {\em Statement on Statistical Significance and Replicability} that appeared in 2021.\footnote{\url{https://magazine.amstat.org/blog/2021/08/01/task-force-statement-p-value/}} As with most official statements, in almost any context, this one says very little, e.g., 
\begin{quote}
{\em Different measures of uncertainty can complement one another; no single measure serves all purposes.}
\end{quote}
While this assertion is politically (and perhaps technically) correct, it offers nothing to help improve the state of affairs.  The lack of any clear guidance in this official statement reveals that there are important and fundamental questions concerning the foundations of statistics and inductive inference that remain unanswered:
\begin{quote}
{\em Should probability enter to capture degrees of belief about claims? ... Or to ensure we won't reach mistaken interpretations of data too often in the long run of experience?} \citep[][p.~xi]{mayo.book.2018}
\end{quote}
The two distinct roles of probability highlighted in the quote above correspond to the classical frequentist and Bayesian schools of statistical inference, which have two fundamentally different priorities, referred to here as {\em performance} and {\em probativeness}, respectively.  Over the last 50+ years, however, the lines between the two perspectives and their distinct priorities have been blurred.  Indeed, both Bayesians and frequentists now focus almost exclusively on performance.  These performance considerations are genuinely important for the logic of statistical inference: 
\begin{quote}
{\em even if an empirical frequency-based view of probability is not used directly as a basis for inference, it is unacceptable if a procedure\ldots of representing uncertain knowledge would, if used repeatedly, give systematically misleading conclusions} \citep[][p.~295]{reid.cox.2014}.
\end{quote}
As the replication crisis has taught us, however, there is more to statistical inference than achieving, say, Type~I and II error probability control.  Beyond performance, we are also concerned with probativeness, i.e., methods' ability to probe for hypotheses that are genuinely supported by the observed data.  Modern statistical methods cannot achieve both performance and probativeness objectives, so a fully satisfactory framework for scientific inferences requires new perspectives. 

Section~\ref{SS:two} gives the problem setup and briefly describes the Bayesian versus frequentist {\em two-theory problem}. There we justify our above claim that modern statistical methods fail to meet both the performance and probativeness objectives.  This includes the default-prior Bayes solution that aims to strike a balance between the two theories.  What holds the default-prior Bayes solution back from meeting the performance and probativeness objectives is its lack of calibration, which is directly related to the constraint that the posterior distribution be a precise probability.  Fortunately, the relatively new, possibility-theoretic {\em inferential model} (IM) framework, reviewed in Section~\ref{SS:ims} below, is able to achieve greater flexibility by embracing a certain type and degree of imprecision in its construction.  We present here a key result, namely, Theorem~\ref{theorem:IMvalidity}, that drives the IM's reliability, even into the new probativeness territory considered here. 

Our main contribution here, in Section~\ref{S:ps}, is a demonstration of the IM's ability to simultaneously achieve both {\em performance} and {\em probabitiveness}.  On the performance side, we show in Section~\ref{SS:perf} that procedures, e.g., hypothesis tests and confidence sets, derived from the IM's necessity and possibility measure output control the frequentist error rates at the nominal level.  Of particular interest is that there are no restrictions on the kinds of questions that the IM can address, so it is at least conceptually straightforward to eliminate nuisance parameters and obtain provably reliable marginal inference.  

We enter new territory in Section~\ref{SS:prob}, where we consider the question of probativeness.  First: {\em what is probing?}  In classical hypothesis testing, typically a null hypothesis is offered and a decision is made to either reject that hypothesis or not.  Often this null hypothesis represents a scientific status quo, e.g., that a new mental health treatment program has no effect on patients' well-being.  Those who follow the mechanical {\em NHST} (null hypothesis significance test) guidelines would believe that all the statistical analysis offers is a reject-or-not decision; in that case, if the investigator's data leads to a reject conclusion, then apparently he/she has made a psychological discovery.  Of course, that logic is flawed because all the statistical test has determined is that the data are incompatible with the status quo.  More specifically, the test does not imply that the data actually support the complementary hypothesis that there is an appreciable benefit to the new treatment, which is the bar for claiming a scientific discovery.  Probing aims to dig deeper than (in)compatibility and look for genuine support.  None of the standard statistical tools offer this probing, so something new is needed.   

Fortunately, possibility measures and imprecise probabilities more generally contain lots of relevant information and, in particular, to each relevant hypothesis it returns a pair of numbers.  As discussed in Section~\ref{SS:perf}, only one of those numbers is used for the usual performance-focused developments.  That is, we reject a hypothesis if its degree of possibility or plausibility is small, since that is an indication of incompatibility.  The other number is commonly understood as measuring a degree of necessity, belief, or support, so a natural question is if this feature of the IM output can be used for probing.  In Section~\ref{SS:prob} we give an affirmative answer to this question and, furthermore, offer some strong theoretical support for the claim that the IM's probing is provably reliable.

The probativeness conclusion is a direct consequence of the IM output's imprecision.  That the additional flexibility of imprecision creates opportunities for more nuanced judgments is one of the motivations for accounting for imprecision, so this is no big surprise.  But our contribution here is valuable for several reasons.  First, the statistical community is aware of this need to see beyond basic performance criteria, but general and easy-to-follow guidance is still lacking.  In Section~\ref{S:mayo} below we summarize a relatively recent proposal in \citet{mayo.book.2018} and compare it to what the IM framework offers.  There we argue that a difficulty with supplementing the standard testing machinery with a probing add-on, as Mayo and others have proposed, is that frequentism lacks an appropriate language to describe anything beyond (in)compatibility.  To clearly articulate what probing or support means, we need a richer language than what frequentist statistics offers.  The possibility-theoretic IM formulation allows for this, but without sacrificing on the frequentist-like performance guarantees.  That is, IMs offer a simpler interpretation based on possibilistic reasoning, where the necessary but complicated frequentist considerations are hidden under the hood in a calibration engine (Theorem~\ref{theorem:IMvalidity}) that powers the IM. Second, our contribution showcases the important role played by imprecise probability, by reinforcing the key point that imprecision is {\em not} due to an inadequacy of the approach, but, rather, is an essential part of a complete and fully satisfactory solution to the statistical inference problem.

Following our comparison of IMs and Mayo's theory of severe testing, we provide several illustrations of the IM solution in Section~\ref{S:examples}, focusing primarily on its probing abilities.  This is followed by some concluding remarks in Section~\ref{S:discuss} and a few relevant details concerning hypothesis testing and connections to the IM theory in Appendix~\ref{app:test.im}. 


\section{Background}
\label{S:background}

\subsection{Two-theory problem}
\label{SS:two}

To set the scene, denote the observable data by $Y$.  The statistical model for $Y$ will be denoted by $\{\prob_\theta: \theta \in \TT\}$ and the unknown true value of the model parameter will be denoted by $\Theta$.  Note that the setup here is quite general: $Y$, $\Theta$, or both can be scalars, vectors, or something else.  We focus here on the typical case where {\em no genuine prior information is available/assumed}.  So, given only the model $\{\prob_{\theta}: \theta \in \TT\}$ and the observed data $Y=y$, the goal is to quantify uncertainty about $\Theta$ for the purpose of making inference.  For concreteness, we will interpret ``making inference'' as making (data-driven) judgments about hypotheses concerning $\Theta$.  In particular, we seek to assign numerical values---could be p-values, posterior probabilities, etc.---to hypotheses $H \subseteq \TT$ concerning $\Theta$ or some feature thereof.  

In a nutshell, the two dominant schools of thought in statistics are as follows. 

\begin{bayespar}
Uncertainty is quantified directly through specification of a prior probability distribution for $\theta$, representing the data analyst's {\em a priori} degrees of belief.  Bayes's theorem is then used to update the prior to a data-dependent posterior distribution for $\theta$.  The posterior probability of a hypothesis $H$ represents the analyst's degree of belief in the truthfulness of $H$, given data, and would be essential for inference concerning $H$.  
That is, the magnitudes of the posterior probabilities naturally drive the data analyst's judgments about which hypotheses are supported by the data and which are not.
\end{bayespar}

\begin{freqpar}
Uncertainty is quantified indirectly through the use of reliable procedures that control error rates. Consider, e.g., a p-value for testing a hypothesis $H$.  What makes such a p-value meaningful is that, by construction, it tends to be not-small when $H$ is true.  Therefore, observing a small p-value gives the data analyst reason to doubt the truthfulness of $H$:
\begin{quote}
{\em The force with which such a conclusion is supported is logically that of the simple disjunction: Either an exceptionally rare chance has occurred, or {\em [the hypothesis]} is not true} \citep[][p.~42]{fisher1973}.
\end{quote}
The p-value {\em does not} represent the ``probability of $H$'' in any sense.  So, a not-small (resp.~small) p-value cannot be interpreted as direct support for $H$ (resp.~$H^c$) or any sub-hypothesis thereof. 
\end{freqpar}

So, at least in principle, the Bayesian framework focuses on probativeness whereas the frequentist framework focuses on performance.  But the line between frequentist and modern Bayesian practice is not especially clear.  Even Bayesians typically assume little or no prior information, as we have assumed here, so default priors are the norm \citep[e.g.,][]{berger2006, jeffreys1946}.  With an artificial or default prior, however, the ``degree of belief'' interpretation of the posterior probabilities is lost, 
\begin{quote}
[Bayes's theorem] {\em does not create real probabilities from hypothetical probabilities} \citep[][p.~249]{fraser.copss}
\end{quote}
and, along with it, the probative nature of inferences based on them,
\begin{quote}
{\em ...any serious mathematician would surely ask how you could use} [Bayes's theorem] {\em with one premise missing by making up an ingredient and thinking that the conclusions of the} [theorem] {\em were still available} \citep[][p.~329]{fraser2011.rejoinder}.
\end{quote}
The default-prior Bayes posterior probabilities could still have performance assurances {\em if} they were suitably calibrated.  But the {\em false confidence theorem} \citep{balch.martin.ferson.2017} shows that this is not the case: there exists false hypotheses to which the posterior tends to assign large probabilities.  In particular, let $\prior_y$ denote a data-dependent probability distribution for $\Theta$, e.g., default-prior Bayes, fiducial, etc.  Then the false confidence theorem states that, for any $(\rho,\tau) \in (0,1)^2$, there exists hypotheses $H \subseteq \TT$ such that 
\[ H \not\ni \Theta \quad \text{and} \quad \prob_\Theta\{ \prior_Y(H) > \tau \} > \rho. \]
This implies that inferences based on the magnitudes of these probabilities---i.e., if $\prior_y(H)$ is small, then infer $H^c$---are at risk of being ``systematically misleading'' (cf.~Reid and Cox).  This explains why modern Bayesian analysis focuses less on probabilistic reasoning based on the posterior probabilities themselves and more on the performance of procedures (tests and credible sets) derived from the posterior distribution. Hence modern Bayesians and frequentists are not so different.  


The key take-away message is as follows.  Pure frequentist methods focus on detecting incompatibility between data and hypotheses (performance), so they do not offer any guidance on how to identify hypotheses actually supported by the data (probativeness).  Default-prior Bayesian methods are effectively no different, so this critique applies to them too.  More specifically, the default-prior Bayes posterior probabilities lack the calibration necessary to reliably check for either incompatibility or support.  Therefore, at least when prior information is vacuous, neither of the mainstream schools of thought in statistics can simultaneously achieve both the performance and probativeness objectives.

\subsection{Inferential models overview}
\label{SS:ims}

The inferential models (IM) framework was first developed in \citet{imbasics, imbook} as a fresh perspective on Fisher's fiducial argument \citep{fisher1935a, zabell1992} and on the Dempster--Shafer theory of belief functions \citep{dempster.copss, dempster1968a, dempster2008, shafer1976} in the context of statistical inference.  It aimed to balance the Bayesians' desire for belief assignments and the frequentists' desire for error rate control.  A key distinction between IMs and the familiar Bayesian and frequentist frameworks is that the output is an {\em imprecise probability} or, more specifically, a {\em necessity--possibility measure} pair.  
\begin{quote}
{\em Possibility is an entirely different idea from probability, and it is sometimes, we maintain, a more efficient and powerful uncertainty variable, able to perform semantic tasks which the other cannot} \citep[][p.~103]{shackle1961}.
\end{quote}
Unlike in other applications of imprecise probability, the imprecision that enters into the picture here is {\em not} the result of the data analyst's inability or unwillingness to precisely specify a statistical model, etc., although partially identified models \citep[e.g.,][]{manski.book} could be a source of additional imprecision.  Instead, it has been shown that a certain degree of imprecision is necessary for inference to be {\em valid} in a specific statistically-relevant sense that we explain below.  Moreover, it has also been shown that a possibility measure is the ``right'' imprecise probability model for quantifying this unique form of imprecision, as opposed to more general belief functions, lower previsions, etc.~that are designed for modeling ignorance-driven imprecision, or Knightian uncertainty.  Below is a quick summary of the IM construction and explanation of the claims just made.  

The IM construction summarized here is the likelihood-driven construction, recently advanced in \citet{ryanpp2}, which is based on a holistic view of the statistical inference problem, i.e., it aims to answer the question {\em what does the data have to say about $\Theta$?}  This is our preferred construction, but it is worth pointing out here that this is not the only available option.  Appendix~\ref{app:test.im} outlines an alternative construction, following \citet{imchar}, that starts with a specific hypothesis testing problem to be solved; since this starts with a specific rather than general question about $\Theta$ to answer, we consider this test-based construction ``less holistic'' than the likelihood-based construction mentioned above and described in more detail below.  The two different constructions have their advantages and, in particular, the test-based construction helps us make connections to earlier attempts to achieve probativeness. 

The likelihood-based construction here is motivated by the probability-to-possibility transform in, e.g., \citet{dubois2004}, \citet{dubois2006}, \citet{hose.hanss.2021}, and \citet{hose2022thesis}, and is driven by the likelihood function of the posited model.  Let $\theta \mapsto L_y(\theta)$ denote the likelihood function for $\Theta$ based on data $y$, and define the {\em relative likelihood}
\[ R(y, \theta) = \frac{L_y(\theta)}{L_y(\hat \theta_y)}, \quad \theta \in \TT,\]
where $\hat \theta_y$ is a maximum likelihood estimator.  This relative likelihood has been used by several authors \citep[e.g.,][]{shafer1982, wasserman1990b, denoeux2014} to construct a plausibility function for $\Theta$.  To achieve the desired performance guarantees, however, we need to go one step further.  Next, define the function 
\begin{equation}
\label{eq:IMcontour}
\pi_y(\theta) = \prob_{\theta}\{ R(Y,\theta) \leq R(y,\theta)\}, \quad \theta \in \TT. 
\end{equation}
This is the p-value function for a likelihood ratio test, but it is also a possibility contour, since it attains a maximum value of 1 at $\hat \theta_y$. Then the corresponding IM for $\Theta$ is the possibility and necessity measure pair determined by the contour in \eqref{eq:IMcontour}, i.e., 
\begin{equation}
\label{eq:IMposs}
\uPi_y(H) = \sup_{\theta \in H} \pi_y(\theta) \quad \text{and} \quad \lPi_y(H) = 1 - \uPi_y(H^c), \quad H \subseteq \Theta. 
\end{equation}
It is easy to verify from the above definition that 
\begin{equation}
\label{eq:order}
\lPi_y(H) \leq \uPi_y(H), \quad \text{for all $H \subseteq \TT$ and all $y \in \YY$}, 
\end{equation}
which explains the lower- and upper-bar notation and why, in some cases, these are referred to as lower and upper probabilities.  

A feature of the IM's output $(\lPi_y, \uPi_y)$, or necessity-possibility measure pairs more generally, is that there are some inherent constraints on the values that $\lPi_y(H)$ and $\uPi_y(H)$ can take for a given $H$.  In particular, 
\begin{equation}
\label{eq:zero.one}
\uPi_y(H) < 1 \implies \lPi_y(H) = 0 \quad \text{and} \quad \lPi_y(H) > 0 \implies \uPi_y(H) = 1. 
\end{equation}
The intuition \citep[cf.~][]{shackle1961} behind this is as follows: if there is any doubt about $H$, so that $\uPi_y(H) < 1$, then there cannot be even a shred of support for $H$ and, similarly, if there is a shred of support for $H$, so that $\lPi_y(H) > 0$, then there can be no doubt that $H$ is possible.  These constraints arise from the definition \eqref{eq:IMposs} that determines $\uPi_y$ by optimizing $\pi_y$. To some these constraints might be a restriction and, indeed, there are some non-statistical applications in which this structure, what Shafer calls {\em consonance}, would not be appropriate. However, our views align with those of Shafer:
\begin{quote}
{\em ... specific items of evidence can often be treated as consonant, and there is at least one general type of evidence that seems well adapted to such treatment.  This is inferential evidence---the evidence for a cause that is provided by an effect} \citep[][p.~226]{shafer1976}.
\end{quote}
Statistical inference problems, like those in consideration here, are of the form Shafer is referring to, so the adoption of a consonant belief structure for quantifying uncertainty is quite natural.  In addition, the particular property (Theorem~\ref{theorem:IMvalidity}) that we need to ensure both performance and probativeness can only be satisfied if the IM has this consonant structure, i.e., if its output is a necessity-possibility measure pair.



Suppose interest is in some feature $\Phi = f(\Theta)$, where $f$ is a function defined on $\TT$. We can easily obtain a marginal IM for $\Phi$ from that for $\Theta$ using possibility calculus.  Indeed, the extension principle of \citet{Zadeh1975} gives the possibility contour function for $\Phi$:
\begin{equation}
\label{eq:featureIMcontour}
\pi_{y}^f(\phi) = \sup_{\theta: f(\theta) = \phi}\pi_{y}(\theta), \quad \phi \in f(\TT).
\end{equation}
Using this contour, just as before, we can obtain the possibility and necessity measure pair that determines the marginal IM for $\Phi$:
\[ \uPi_{y}^f(K) = \sup_{\phi \in K} \pi_{y}^f(\phi) \quad \text{and} \quad \lPi_y^f(K) = 1 - \uPi_y^f(K^c), \quad K \subseteq f(\TT).\] 
In Section~\ref{S:ps} below, we demonstrate that this marginal IM inherits the desirable properties of the original IM construction for $\Theta$. This characteristic is crucial as it safeguards data analysts from unintentional errors in downstream or subsequent inferences. But this is not the only marginalization strategy. The one above is consistent with our holistic approach to/perspective on statistical inference, but the price paid for its broad flexibility is efficiency.  If it were known that {\em only} the feature $\Phi = f(\Theta)$ is of interest, then a different and more efficient marginalization strategy can be carried out, one that is tailored specifically to that feature; see, e.g., \citet{ryanpp2}. 

A relevant question is: {\em how to interpret the IM output?}  Since the IM output corresponds to an imprecise probability, all the standard interpretations of imprecise probabilities can be taken, e.g., degrees of belief, bounds on prices for gambles, etc.  In particular, for fixed $y$, the IM output $(\lPi_y, \uPi_y)$ defines a coherent lower and upper probability for $\Theta$.  While IMs are compatible with the theory developed in \citet{walley1991}, that is not the perspective we take here.  We should also emphasize that there does not exist an underlying ``true conditional probability distribution of $\Theta$, given $Y=y$,'' so it would not make sense to think of these imprecise probabilities as bounds on probabilities on some ``true'' probabilities, or that this ``true'' probability distribution is contained in the IM output's credal set.  Instead, we see the IM output as facilitating what we call {\em possibilistic reasoning}---a sort of unidirectional version of the more familiar, bidirectional probabilistic reasoning.  That is, in the latter case, both small and large probability values carry inferential weight whereas, in the former case, only small possibility values and large necessity values carry inferential weight.  To conclude that data $y$ supports a hypothesis $H$, it is not enough that $\uPi_y(H)$ is large; we need $\lPi_y(H)$ to be large, which implies that $\uPi_y(H)$ is also large, by \eqref{eq:order}.  In fact, \citet[][Ch.~11]{shafer1976} refers to $H \mapsto \lPi_y(H)$ as a {\em support function}, which is how we propose to use it here.  This is just a mathematization of the commonsense notion that a lack of support for $H$ does not imply support for $H^c$.  

A certain mathematical structure is not enough to give the IM output the aforementioned ``inferential weight.''  Following Cournot's principle \citep[e.g.,][]{shafer2007}, this requires establishing that true hypotheses tend not to be assigned small possibility values; equivalently, false hypotheses tend not to be assigned large necessity values.  The following basic but important result establishes a key connection between the IM and the ``real world'' (relative to the posited model), through the magnitudes of its possibility assignments to true hypotheses.  This will serve as the jumping off point for both the performance- and probativeness-specific properties in the coming section.  

\begin{thm}
\label{theorem:IMvalidity}
An IM for $\Theta$ whose output takes the form of a necessity-possibility measure pair as described above, determined by a contour function $\pi_y(\theta)$ as in \eqref{eq:IMcontour}, is {\em strongly valid} in the sense that 
\begin{equation}
\label{eq:valid}
\sup_{\Theta \in \TT} \prob_\Theta\{ \pi_Y(\Theta) \leq \alpha \} \leq \alpha, \quad \alpha \in [0,1]. 
\end{equation}
\end{thm}


If one is thinking in terms of p-values, then the result in Theorem~\ref{theorem:IMvalidity} will look familiar.  It also closely resembles what \citet{walley2002} calls the {\em fundamental frequentist principle}, so, despite its familiarity, this result must be important.  We will discuss below, in Section~\ref{S:ps}, the striking implications this has when it comes to the IM's performance and probativeness properties.

\section{Two P's in a possibility-theoretic pod}
\label{S:ps}

\subsection{Performance}
\label{SS:perf}

As discussed above, what modern statisticians value most is {\em performance}, i.e., that the procedures developed for the purpose of making inference-related decisions (e.g., accept or reject a hypothesis) have frequentist error rate control guarantees.  These error control properties are genuinely important: if statistical methods are not even reliable, then it is difficult to imagine how they could help advance science.  This explains why even the Bayesians are concerned with frequentist properties.  

Most of the previous IM developments have focused primarily on the performance aspect, so the results presented below are not new.  For completeness, however, we give a quick summary of the available results and offer some new perspectives.  The two standard procedures found in the statistics literature are hypothesis testing and confidence set procedures.  Corollary~\ref{cor1} below describes the corresponding procedures derived from the IM output and the error rate control guarantees they enjoy.  

\begin{cor}\label{cor1}
Consider an IM for $\Theta$ that, for $Y=y$, returns the necessity-possibility measure pair $(\lPi_y, \uPi_y)$ determined by a possibility contour function $\pi_y$ as described in Section~\ref{SS:ims}.  Then the following properties hold for all $\alpha \in [0,1]$. 
\begin{enumerate}[(a)]
\item For any given $H$, the test ``reject $H$ if and only if $\uPi_Y(H) \leq \alpha$'' has frequentist Type~I error probability no more than $\alpha$, i.e., 
\[ 
\sup_{\Theta \in H} \prob_\Theta\{ \uPi_Y(H) \leq \alpha \} \leq \alpha. \]
\item The set $C_\alpha(Y) = \{\theta: \pi_Y(\theta) > \alpha\}$ has frequentist coverage probability at least $1-\alpha$, making it a $100(1-\alpha)$\% confidence set. That is,
\[ 
\sup_{\Theta \in \TT} \prob_\Theta\{ C_\alpha(Y) \not\ni \Theta \} \leq \alpha. \]
\end{enumerate}
\end{cor}

\begin{proof}
Both of these results are immediate consequences of \eqref{eq:valid}.  For Part~(a), note that monotonicity of the possibility measure implies that $\uPi_Y(H) \geq \pi_Y(\Theta)$ for any $\Theta \in H$.  Therefore, combined with \eqref{eq:valid}, we get 
\begin{equation}
\label{eq:valid.up}
\prob_\Theta\{ \uPi_Y(H) \leq \alpha \} \leq \prob_\Theta\{ \pi_Y(\Theta) \leq \alpha \} \leq \alpha. 
\end{equation}
For Part~(b), observed that $C_\alpha(Y) \not\ni \Theta$ if and only if $\pi_{Y}(\Theta) \leq \alpha$.  And since the latter event has probability $\leq \alpha$ by \eqref{eq:valid}, so too does the former.  
\end{proof}

We are not aware of Fisher ever making such a statement, but we can imagine that Fisher's disdain for the Neyman-style behavioral approach to statistical inference at least partially stemmed from the fact that the frequentist error rate control properties would be immediate consequences of the kind of calibration needed to make his ``logical disjunction'' argument sound.  That is, if Fisher's necessary calibration \eqref{eq:valid} is satisfied, then Neyman's error rate control is a corollary.  This is effectively what Corollary~\ref{cor1} shows and, it is in this sense that a strongly valid IM offers performance guarantees.  

Recall that we explained in Section~\ref{SS:ims} how uncertainty about a relevant feature $\Phi = f(\Theta)$ could be quantified based solely on the IM for $\Theta$.  An immediate consequence of Corollary~\ref{cor1} and the possibility calculus---the extension principle specifically---is that the corresponding test and confidence set procedures for making decisions pertaining to $\Phi$ inherit the performance guarantees that the IM for $\Theta$ enjoys.  

The kind of performance properties that the IM achieves might remind some readers of confidence distributions \citep[e.g.,][]{xie.singh.2012, schweder.hjort.2002, nadarajah.etal.2015}.  For a scalar parameter $\Theta$, a confidence distribution is a data-dependent cumulative distribution function $\theta \mapsto G_y(\theta)$ such that 
\[ \sup_{\Theta \in \TT} \prob_\Theta\{ G_Y(\Theta) \leq \alpha \} \leq \alpha, \quad \alpha \in [0,1]. \]
From here, one can construct hypothesis tests and confidence sets similar to how we did with the IM output above in Corollary~\ref{cor1}; see the above references for details.  However, the testing error rate control can only be achieved for hypotheses $H$ that take the form of half-lines, e.g., $(-\infty, \theta]$ or $[\theta, \infty)$.  For other kinds of hypotheses, e.g., bounded intervals, the frequentist error rates might not be controlled.  Corollary~\ref{cor1} shows that the IM controls error rates for any hypotheses $H$, and not just for scalar $\Theta$.  For a certain class of models, Fisher's fiducial distribution and the default-prior Bayes posterior distribution are confidence distributions, and \citet{Martin2023} characterizes these as members of the IM output's credal set.  This characterization explains why a confidence distribution's probability assignments are calibrated only in the tails, i.e., for half-line hypotheses.  



\subsection{Probativeness}
\label{SS:prob}

The current literature on IMs has focused largely on the performance-related questions as in the previous subsection.  This is understandable given that performance is the top priority for modern statisticians and that other performance-related features (e.g., Theorem~\ref{theorem:IMvalidity}) are crucial to Fisher's brand of inductive inference.  But we claim that the IMs described above have even more to offer, so the goal of this section is to unearth those previously underappreciated features of the IM framework.  

That the IM output offers more than what has been discussed in the extant literature is obvious: the focus has been on the performance of derived statistical methods, which only involves certain features, such as the contour function, its level sets, and the possibility measure evaluated at pre-determined hypotheses.  This is just a small fraction of what a full-blown imprecise probability distribution---or even a necessity-possibility measure pair---can do.  What we are particularly interested in here is the use of the IM output to {\em probe}, that is, to naturally proceed with the analysis, to dig deeper, after an answer to the first (often trivial) question has been given.  On the performance side, one thinks of the test in Corollary~\ref{cor1}(a) as a one-and-done prospect: if $H$ is rejected, then infer $H^c$ and pack up to go home.  In reality, such a test is just the first step in the analysis, so we ought to consider the follow-up questions and analyses too.  This is especially true in the IM case because there is an opportunity to tap into those  aspects of the necessity-possibility measure pair that are currently being ignored.  


Consider the {\em common} situation where the data $y$ is incompatible with the hypothesis $H$ in the sense that $\uPi_y(H)$ is small; the other case of probing when $\uPi_y(H)$ is not small is more challenging and will be discussed in detail in Section~\ref{S:mayo}.  We emphasized ``common'' because the initial $H$, or {\em null hypothesis}, is often an overly simplistic scientific default that isn't expected to be true---otherwise, the resources needed to collect the data $y$ probably would not have been invested.  If $\uPi_y(H)$ is small, then we know by the discussion in Section~\ref{SS:ims} that $\uPi_y(H^c) = 1$ and, consequently, there is ample room for certain sub-hypotheses in $H^c$ to have non-trivial necessity values, i.e., $\lPi_y(A) > 0$ for some $A \subseteq H^c$.  We follow Shafer and interpret $\lPi_y$ as a measure of support, so this probing exercise is about finding sub-hypotheses whose truthfulness is directly supported by data $y$.  To fix ideas, consider the case where $\uPi_y(H)$ is small and $A$ is a fixed sub-hypothesis contained in $H^c$.  There are roughly two cases worth exploring:
\begin{itemize}
\item if $\lPi_y(A)$ is large, then we can conclude that $A$ is supported by data $y$, and 
\vspace{-2mm}
\item if $\lPi_y(A)$ is small, then data $y$ is mostly uninformative about $A$ and no conclusion about $A$ is warranted; that is, both $A$ and $A^c$ are compatible with $y$ or, equivalently, the ``don't know probability'' \citep{dempster2008}, $\uPi_y(A) - \lPi_y(A)$, is large. 
\end{itemize} 
This process can be repeated, in principle, for all sub-hypotheses $A$ of $H$.  The data analyst will find some $A$s that are supported by data and others about which the data are mostly uninformative.  Stitching all of these $A$-specific analyses together creates a complete IM tapestry that details what the data can reliably say about $\Theta$.  

\citet[][p.~13 and elsewhere]{mayo.book.2018} argues at a high level that ``probabilism'' does not imply probativeness.  But the shortcoming of probability as a tool for probing also becomes clear here in the mathematical details.  Take, for example, a confidence distribution as discussed briefly above.  Since the probabilities assigned by the confidence distribution to $H$ and $H^c$, respectively, must sum to 1, we find that a lack of support for one implies support for the other, which we know is logically incorrect---this is exactly why the probing task is challenging.  Therefore, imprecision seems necessary to achieve the probing goal, and below we will argue why the IM framework suggested here is the appropriate formulation. Although Mayo does not specifically mention imprecise probabilities,\footnote{In the footnote on page 67, Mayo presents Popper's argument for how there can be logical contradictions if Carnap's ``degree of confirmation'' is measured by something additive like a probability.} we will show below that the measure she proposes is, in fact, non-additive and has connections with our proposed IM solution in the contexts where she described it; see Section~\ref{S:mayo}. 

That the IM framework facilitates probing as described above does not directly imply that the probing process we described is {\em reliable}.  We do get some comfort from \eqref{eq:valid.up}, i.e., the possibilities assigned to true hypotheses do not tend to be small.  Thanks to the duality $\lPi_y(H) = 1-\uPi_y(H^c)$ between the two measures, this also implies  
\begin{equation}
\label{eq:valid.lo}
\sup_{\Theta \not\in H} \prob_\Theta\{ \lPi_Y(H) \geq 1-\alpha \} \leq \alpha, \quad \text{all $\alpha \in [0,1]$, all $H \subseteq \TT$}. 
\end{equation}
That is, the necessities (or support) assigned to false hypotheses do not tend to be large.  It is problematic, however, that probing is dynamic and there is no way to predict what kind of follow-up questions the data analyst might want to ask.  In fact, those questions might be determined by the data itself, so the aforementioned comfort---derived from hypothesis-wise error rate control---is not all that comforting. \citet[][Sec.~4.2]{mayocox2006} discuss these and related issues concerning selection.  Fortunately, there are stronger consequences of Theorem~\ref{theorem:IMvalidity} that are not captured by \eqref{eq:valid.up} and have not been elucidated in the previous works on IMs and their performance properties.  

\begin{cor}
\label{cor2}
An IM for $\Theta$ whose output $(\lPi_y,\uPi_y)$ is determined by a possibility contour $\pi_y$ via \eqref{eq:IMposs} has the following {\em uniform validity} property:  
\begin{equation}
\label{eq:uniform}
\sup_{\Theta \in \TT} \prob_\Theta\{ \text{$\uPi_Y(H) \leq \alpha$ for some true $H$, i.e., $H \ni \Theta$} \} \leq \alpha, \quad \alpha \in [0,1]. 
\end{equation}
Equivalently, in terms of necessity/support:
\[ \sup_{\Theta \in \TT} \prob_\Theta\{ \text{$\lPi_Y(H) \geq 1-\alpha$ for some false $H$, i.e., $H \not\ni \Theta$} \} \leq \alpha, \quad \alpha \in [0,1]. \]
\end{cor}

The proof is almost immediate so we get it out of the way now. But the reader might want to first skip ahead to the discussion below to better understand the result.

\begin{proof}[Proof of Corollary~\ref{cor2}]
The claim \eqref{eq:uniform} follows from \eqref{eq:valid} and the fact that there exists an $H$ such that $H \ni \Theta$ and $\uPi_Y(H) := \sup_{\theta \in H} \pi_Y(\theta) \leq \alpha$ if and only if $\pi_Y(\Theta) \leq \alpha$. 
\end{proof}

The ``for some true $H$'' statement in \eqref{eq:uniform} is potentially confusing so here is a more detailed explanation.  The reader is surely accustomed to interpreting ``for some'' in terms of a union operation, and that is precisely the interpretation we have in mind here.  That is, the event in \eqref{eq:uniform} can be written as 
\[ \bigcup_{H \subseteq \TT: H \ni \Theta} \{ \uPi_Y(H) \leq \alpha \}. \]
This is clearly a much larger event than $\{\uPi_Y(H) \leq \alpha\}$ for a fixed $H$ and, therefore, the probability bound in \eqref{eq:uniform} is significantly stronger than the analogous bound in \eqref{eq:valid.up}.  Aside from being mathematically stronger, there are key practical implications of this.  The result implies that no matter how the data analyst proceeds with his/her probing, the probability that even one of the IM's suggestions is misleading---small $\uPi_Y$ to a true hypothesis or large $\lPi_Y$ to a false hypothesis---is controlled at the specified level.  To put the result more in line with the explanation of probing given above, consider the special case of a fixed $H$ and then a collection of sub-hypotheses $\{A_r: r \geq 1\}$ of $H^c$.  So, if $\Theta \in H$, then $\Theta \in A_r^c$ for $r \geq 1$. 
 Then Corollary~\ref{cor2} implies that 
\[ \sup_{\Theta \in H} \prob_\Theta\{ \text{$\uPi_Y(H) \leq \alpha$ or $\lPi_Y(A_1) \geq 1-\alpha$ or $\lPi_Y(A_2) \geq 1-\alpha$ or...} \} \leq \alpha, \]
i.e., the probability that the IM points the data analyst in the wrong direction concerning even one of these hypotheses is no more than $\alpha$.  Moreover, that union perspective reveals that the result remains true even if the hypotheses chosen in the probing step happen to be data-dependent in some way that would be too complicated for those of us here on the data analysis sidelines to specify in advance.  So, the uniformity baked into the result of Corollary~\ref{cor2} is exactly what is needed to ensure that the commonsense probing that the IM framework suggests can indeed be carried out reliably.

\section{Comparison with Mayo's severity}
\label{S:mayo}

\subsection{Background}
\label{SS:mayo.back}

In Section~\ref{S:intro}, we mentioned recent efforts by statisticians to supplement the standard significance tests, etc.~with measures designed to {\em probe} for hypotheses that are supported by the data.  In particular, what \citet{mayo.book.2018} refers to as {\em severe testing} aims to capture this notion of probing.  We do not assume that the reader is familiar with Mayo's work, so here we give a relatively brief introduction.  Modulo some minor changes in notation and terminology, here is how \citet[][p.~23]{mayo.book.2018} explains her notion of severity:
\begin{quote}
{\em Severity (weak):} If data $y$ agree with a claim $H$ but the method was practically incapable of finding flaws with $H$ even if they exist, then $y$ is poor evidence of $H$. \par 
{\em Severity (strong):} If $H$ passes a test that was highly capable of finding flaws or discrepancies from $H$, and yet none or a few are found, then the passing result, $y$, is an indication of, or evidence for, $H$. 
\end{quote}
At least conceptually, we think most readers would find these basic severity principles uncontroversial.  The idea goes back to Popper's falsificationist program which says science progresses by subjecting the status quo to severe tests, tests that are capable of detecting departures from the status quo.  Hypotheses that are able to withstand a series of severe tests have ``proved their mettle'' \citep[][p.~10]{popper1959}.  The challenge in the context of statistical hypothesis testing is that Popper's very strict standard for falsification cannot be met based on a (necessarily limited) set of empirical data $y$.  This is where the work of Fisher, Neyman--Pearson, Mayo--Cox, and others comes in.  

As we attempt to dig deeper, to get beyond just a high-level conceptual understanding of these ideas in hopes of putting them into practice, things become less clear.  \citet[][p.~148--150]{mayo.book.2018} presents her two-part {\em Principle of Frequentist Evidence} (FEV), whose starting point is a particular null hypothesis $H_0$ and a given procedure for testing that hypothesis based on data $y$.  We present here an inconsequentially modified version of FEV from
\citet[][p.~82--84]{mayocox2006}:
\begin{quote}
{\em FEV 1.} $y$ is (strong) evidence against $H_0$, i.e., (strong) evidence of a discrepancy from $H_0$, if and only if, where $H_0$ a correct description of the mechanism generating $y$, then, with (very) high probability, this would have resulted in a less discordant result than is exemplified in $y$. \par 

{\em FEV 2.} A moderate p-value is evidence of the absence of a discrepancy $\delta$ from $H_0$, only if there is high probability the test would have given a worse fit with $H_0$ (i.e., smaller p-value) were a discrepancy $\delta$ to exist. 
\end{quote}
As \citet[][p.~149]{mayo.book.2018} admits, ``this sounds wordy and complicated.'' The complication, we claim, stems from trying to justify fixed-data conclusions based on frequentist-style performance probabilities. On the one hand, FEV1 is familiar even if it is not easy to follow: this is just a different way to say that a small p-value is interpreted as evidence in $y$ against $H_0$.  In this case, the data analyst is in a situation like that described in Section~\ref{SS:prob} and the probing question concerns support for subsets of $H_0^c$.  On the other hand, FEV2 goes in an unfamiliar---but important, challenging, and exciting---direction, towards something far beyond what one finds in the cookbook-style NHST literature.  For both parts of FEV, something more than just the p-value associated with the pair $(y,H_0)$ is needed, something that can probe for genuine support.  What this ``something more'' might look like is discussed below. 

Note that FEV1 and FEV2 are ``if and only if'' and ``only if'' claims, respectively.  The point is that effectively the only way we can interpret incompatibility or discordance between $y$ and $H_0$ is as evidence against $H_0$, but compatibility alone between $y$ and $H_0$ need not be evidence supporting $H_0$.  There are many reasons why a p-value might be large, reasons that do not indicate genuine support in the data for the hypothesis.  This aligns with our aforementioned commonsense understanding, which draws a first connection between what Mayo aims to achieve and what our proposed IM offers.  

After a preview in Chapter~3, Chapter~5 of \citet{mayo.book.2018} lays out some details of her proposal for what the aforementioned ``something else'' ought to look like.  She explains that the origins of her idea are in the early works on power analysis, in particular, \citet{Neyman1955} and insights sprinkled throughout \citet{cox_2006}.  Like p-values can be interpreted as an {\em attained significance level}---we used this interpretation in the procedure-driven IM construction, in particular, Equation~\eqref{eq:contour.test} in Appendix~\ref{app:test.im1}---one can consider a corresponding {\em attained power}.  Mayo's presentation focuses solely on simple, albeit important/common testing scenarios, so we will do the same here, both for concreteness and to avoid potentially misrepresenting Mayo's proposal.  

Suppose $\Theta$ is a scalar parameter and, without loss of generality, consider the null hypothesis $H_0: \Theta \leq \theta_0$, for a fixed $\theta_0$; the opposite one-sided hypothesis $H_0: \Theta \geq \theta_0$ can be handled similarly.  Let $S(Y,\theta_0)$ be a test statistic such that large values indicate incompatibility between $Y$ and $H_0$ and suggest rejection.  More specifically, a test that controls the Type~I error at a specified significance level $\alpha \in (0,1)$ rejects $H_0$ if and only if $S(Y,\theta_0) \geq s_\alpha$, where the critical value $s_\alpha$ is defined to satisfy the condition
\[ \sup_{\theta \leq \theta_0} \prob_\theta\{ S(Y, \theta_0) \geq s_\alpha \} = \alpha. \]
It will often be the case that the probability in the above display is increasing in $\theta$, which implies the supremum is attained at the boundary $\theta_0$.  We will assume that this is the case here, so $s_\alpha$ satisfies $\prob_{\theta_0}\{ S(Y,\theta_0) \geq s_\alpha \} = \alpha$. 
Note that the Type~I error probability is a property of the test procedure and, therefore, does not depend on the observed data $Y=y$.  The p-value, or attained significance level, replaces the fixed critical value $s_\alpha$ with the value of the test statistic computed at the observed data $y$:
\begin{equation}\label{eq:pval.sev}
\pval^{\leq}_y(\theta_0) = \prob_{\theta_0}\{ S(Y,\theta_0) \geq S(y,\theta_0)\} \}.
\end{equation}
The superscript ``$\leq$'' is to emphasize that this p-value is associated with the left-sided null hypothesis $H_0: \Theta \leq \theta_0$. Note that the p-value in \eqref{eq:pval.sev} can be extended to a function $\pval^{\leq}_y(\theta) = \prob_\theta\{ S(Y,\theta) \geq S(y,\theta)\}$ that takes any value $\theta$ as its argument, not just $\theta_0$. 


\subsection{Severity measure}

Consider a claim or hypothesis $H$ about the relevant unknowns, and the goal is to assess the extent to which this hypothesis is supported by the observed data.  Mayo proposes a measure that she calls {\em severity} which is a data-dependent function that maps $H$ to values in $[0,1]$.  In words, Mayo's severity measure at the hypothesis $H$ is defined as 
\begin{equation}\label{eq:Mayosev}
\text{{\em the probability of a worse fit between data and $H$ if $H$ is false.}  }  
\end{equation} 
A high severity value indicates that $H$ has been severely tested, meaning that the data strongly support the truthfulness of $H$.

In practice, the decision to assess a particular claim may arise through various means. Mayo considers a very common scenario where the starting point is a null hypothesis that is tested, and rejection or non-rejection of this null hypothesis determines what follow-up claims might be considered. Throughout this section we focus on an initial null hypothesis of the form $H_0: \Theta \leq \theta_0$. On the one hand, if $H_0$ is rejected, the data analyst is in a situation like that described in Section~\ref{SS:prob} and the probing question concerns support for subsets of $H_0^c$, i.e., claims of the form ``$\Theta > \theta$'' for some $\theta > \theta_0$. On the other hand, if $H_0$ is not rejected, then probing questions would concern supersets of $H_0$, i.e., claims of the form ``$\Theta \leq \theta$'' for $\theta > \theta_0$.  We consider these two cases separately below.

 
\begin{description}
\item[Case 1:] {\em Small p-value, probing for support in (subsets of) the alternative}. 

If the p-value for $H_0: \Theta \leq \theta_0$ is small, then we might be inclined to reject $H_0$.  But does the data genuinely support any subsets of the alternative?  To answer this question, we probe by considering various subsets of the alternative, e.g., ``$\Theta > \theta$'' for $\theta>\theta_0$.  In this case, Mayo's severity measure \eqref{eq:Mayosev} specializes to
\begin{align*}
\sev_y(\{\Theta > \theta\}) &= \inf_{\vartheta \leq \theta} \prob_\vartheta\{ S(Y,\theta) < S(y, \theta) \}  \\
    &=1-\sup_{\vartheta \leq \theta} \prob_\vartheta\{ S(Y,\theta) \geq S(y, \theta) \} \\
    &= 1-\pval^{\leq}_y(\theta), \quad \theta > \theta_0.
\end{align*}
Large values of $\sev_y(H)$ are to be interpreted as stronger support in $y$ for the claim $H$.  Since the right-hand side of the above display is a decreasing function of $\theta$, we get the intuitive property that bolder claims are given less support from the data.  

\item[Case 2:] {\em Not-small p-value, probing for support in (supersets of) the null}. 

If the p-value for $H_0: \Theta \leq \theta_0$ is not small, then we would be inclined to tentatively accept the null.  But despite the null or non-significant conclusion, there is still information available in the data about $\Theta$---``no evidence of risk is not evidence of no risk'' \citep[][p.~3]{mayo.book.2018}. In particular, there are claims more inclusive of the null, i.e., $H$ that are implied by $H_0$, that the data might support, e.g., ``$\Theta \leq \theta$'' for $\theta>\theta_0$. Following \eqref{eq:Mayosev}, severity specializes in this case to 
\begin{align*}
  \sev_y(\{\Theta \leq \theta\}) &= \inf_{\vartheta > \theta} \prob_\vartheta\{ S(Y,\theta) > S(y, \theta) \} \\
    &=1-\sup_{\vartheta > \theta} \prob_\vartheta\{ S(Y,\theta) \leq S(y, \theta) \} \\
    &= 1-\pval^{\geq}_y(\theta), \quad \theta > \theta_0. 
\end{align*}
In this case, severity is associated with the p-value for a right-sided null hypothesis, that is, a hypothesis of the form $\Theta \geq \theta$. Again, large values of $\sev_y(H)$ are to be interpreted as stronger support in $y$ for $H \supseteq H_0$. The right-hand side above is increasing in $\theta$, so less-bold claims are more strongly supported by the data.  
\end{description} 


To summarize, Mayo's severity measure is derived from p-value functions determined by the direction of the claims of interest. A practical way to see it is by treating the complement of the claim of interest, $H^c$, as a null hypothesis, so the severity of $H$ is 1 minus the associated p-value:
\begin{quote}
{\em The severity computation [...] could be arrived at through other means, including varying the null hypothesis.} \citep[][p.~346]{mayo.book.2018}
\end{quote}
One can plot the severity function to visualize the details explained above (see Section~\ref{S:examples}), and Mayo refers to these as {\em severity curves}.

\subsection{IMs versus severity}
\label{ss:IMvsSev}

Having introduced Mayo's severity proposal, the goal now is to compare it to the IM proposal presented above.  Starting with the same setup as in the previous subsection, the aforementioned test-based IM construction is tailored to a specific form of hypothesis the data analyst is interested in evaluating.  Therefore, for the special class of problems involving half-line hypotheses, like Mayo exclusively considers, we can mimic her severity construction above and construct an IM based on the most powerful test of a particular null hypothesis. For example, for a left-sided null hypothesis, it can be shown (see Appendix~\ref{app:mayo}) that this construction determines a possibility contour for $\Theta$, 
\begin{equation}
\label{eq:mayo.im.special}
\pi_y(\theta) = \pval_y^\leq(\theta), 
\end{equation}
that corresponds to the p-value function in \eqref{eq:pval.sev}.  In this context, the p-value function above is monotone increasing, hence the corresponding possibility measure $\uPi_y$ satisfies 
\[ \uPi_y(\{\Theta \leq \theta\}) = \pval_y^\leq(\theta) \quad \text{and} \quad \uPi_y(\{\Theta > \theta\}) = 1, \quad \theta \in \RR. \]
By conjugacy, the corresponding necessity measure $\lPi_y$ satisfies 
\[ \lPi_y(\{\Theta \leq \theta\}) = 0 \quad \text{and} \quad \lPi_y(\{\Theta > \theta\}) = 1-\pval_y^\leq(\theta), \quad \theta \in \RR. \]
To compare ours and Mayo's solutions, we consider the two cases as above separately. 
\begin{description}
\item[Case 1.] {\em Small p-value, probing for support in (subsets of) the alternative}.  

In this case, $\pval_y(\theta_0)$ is small, so probing entails considering hypotheses of the form $\{\Theta > \theta\}$ for $\theta > \theta_0$.  As shown above, 
\[ \sev_y(\{\Theta > \theta\}) = 1 - \pval_y^\leq(\theta) = \lPi_y(\{\Theta > \theta\}), \quad \theta > \theta_0. \]
Therefore, ours and Mayo's probing solutions exactly agree. 
\item[Case 2.] {\em Not-small p-value, probing for support in (supersets of) the null}. 

In this case, $\pval_y(\theta_0)$ is not small, so probing entails considering hypotheses of the form $\{\Theta \leq \theta\}$ for $\theta > \theta_0$.  As shown above, 
\[ \sev_y(\{\Theta \leq \theta\}) = 1 - \pval_y^\geq(\theta) > 0 = \lPi_y(\{\Theta \leq \theta\}), \quad \theta > \theta_0. \]
Therefore, ours and Mayo's probing solutions disagree, and her severity measure dominates the IM's support measure. 
\end{description}
The take-away message is that Mayo's severity solution agrees with the IM solution in Case~1 but not in Case~2.  This difference, of course, deserves some explanation.  

At least at first glance, Mayo's approach is appealing: {\em both} the initial assessment and the follow-up probing assessments are based on use of a most powerful test.  Under ordinary circumstances, few statisticians would find issues concerning the use of a most powerful test---but these are not ``ordinary circumstances.''  Ours and Mayo's goal here is beyond what the classical theory was designed for, i.e., to probe for support in certain follow-up hypotheses depending on the outcome of the initial test, so further consideration is needed.  Indeed, that we are even entertaining the idea of probing implies that we aim to do more than test a hypothesis, so it is not obvious that a standard hypothesis test is the correct starting point.  If one chooses to start with classical hypothesis testing, and if error rate control in the entire probing procedure is desired (see below), then it must be treated more-or-less like a {\em multiple testing} problem.  We say ``more-or-less'' because typical multiple testing has a fixed collection of hypotheses being tested simultaneously, whereas probing ought to allow the investigator to peek at the data before deciding on what other hypotheses besides $H_0$ to consider.  The only probing scheme that 
\begin{itemize}
\item[(a)] rejects $H_0$ if the p-value in \eqref{eq:pval.sev} no more than $\alpha$, and 
\vspace{-2mm}
\item[(b)] controls the overall probing error rate at $\alpha$ over any user probing policy
\end{itemize} 
is one that, when $\pval_y^\leq(\theta_0)$ is not small, assigns support 0 to all hypotheses $H$ that are supersets of $H_0$.  That is precisely what the IM solution does: it is not a ``conservative'' solution, it is the only way to achieve the dual objectives (a) and (b) above.  

It remains to justify objective (b), the desirability of error rate control of the entire probing scheme, in the spirit of Corollary~\ref{cor2}. Mayo and her error-statistician followers adhere to the belief that justification of conclusions about unknowns, drawn through an application of a statistical method to real-world observed data, is directly tied to existence of a proof that the statistical method being employed controls error probabilities relative to the model posited for observable data.  If the error-statisticians prioritizes control of the Type~I error of the initial test, then why would overall probing error control for the entire probing scheme, as in \eqref{eq:uniform}, not be an equally high---if not higher---priority?  

There is, of course, a partial connection between our proposed IM solution (based on a given test) and Mayo's severe testing: they exactly agree in the typical Case~1 where the data is incompatible with the null hypothesis and the goal is to probe for support in the alternative.  It is only in Case~2 where the solutions diverge, but this is the more challenging case where the data is not incompatible with the null.  In this Case~2, Mayo's solution is greedy in the sense that it (implicitly) introduces new hypotheses and the corresponding hypothesis-specific tests with no adjustments to ensure that the probing scheme is reliable.  It is only after making these adjustments that a direct correspondence with our proposed IM solution can be made in Case~2.  From a higher level perspective, we think it should also be of general interest to the imprecise probability community that certain kinds of imprecision are necessary in order to achieve the performance and probativeness properties that statisticians and scientists want and need.

Finally, we must admit that the IM solution presented above is not fully satisfactory.  Even though assigning support 0 for all probing-relevant hypotheses in the Case~2 is unavoidable in some sense, the goal of course is to be more efficient.  Towards this, we remind the reader that the comparison just made between the two solutions is one that takes Mayo's starting point as our starting point.  We are not obligated, however, to start with given hypothesis testing procedure and, in fact, it is the holistic perspective described in Section~\ref{SS:ims} above that allows us, among other things, to overcome the trivial support assignments in Case~2 without sacrificing reliable probing.   

Several illustrations of this are provided in Section~\ref{S:examples} below and, in particular, a direct comparison of the holistic and test-based IM solutions, along with Mayo's severe-testing solution, is given in Section~\ref{SS:mean}.



\section{Illustrations}
\label{S:examples}



\subsection{Normal mean}
\label{SS:mean}

\citet[][p.~142]{mayo.book.2018} describes a hypothetical water plant where the water it discharges is intended to be roughly 150 degrees Fahrenheit.  More specifically, water temperature measurements are assumed to be normally distributed with mean $\Theta$ degrees and standard deviation 10 degrees and, under ideal conditions, $\Theta$ is no more than 150 degrees.  To test the water plant's settings, a sample $Y=(Y_1,\ldots,Y_n)$ of $n=100$ water temperature measurements are taken.  Then the sampling distribution of the sample mean, $\bar Y$, is $\nm(\Theta, 1)$.  Since water temperatures higher than 150 degrees might harm the ecosystem, of  interest is testing the null hypothesis $H_0: \Theta \leq \theta_0$, where $\theta_0 = 150$, versus the alternative $H_1: \Theta > \theta_0$.  After this primary question is addressed, we have the option to probe other hypotheses of the form $(-\infty, \theta]$ or $(\theta, \infty)$, for $\theta$ near 150. 

A most powerful test is available in this example, and it rejects $H_0: \Theta \leq \theta_0$ when $\bar Y - \theta_0$ is large.  Then it follows easily that the p-value function is 
\[ \pval_y^\leq(\theta) = 1 - {\tt pnorm}(\bar y - \theta), \quad \Theta \in \RR, \]
where {\tt pnorm} denotes the standard normal distribution function. As discussed above, this p-value function determines both the IM and Mayo's severity.  



Suppose we observe $\bar y=152$, which is potentially incompatible with the null hypothesis $H_0: \Theta \leq 150$.  Indeed, a plot of $\theta \mapsto \pval_y^\leq(\theta) = \uPi_y(\{\Theta \leq \theta\})$ is shown in Figure~\ref{fig:mean}(a) and we see that, at $\theta = \theta_0 = 150$, the possibility is smaller than 0.05, so we would be inclined to reject the null hypothesis. To probe for support of subsets of the alternative hypothesis, we also plot the severity/necessity 
\[ \lPi_y(\{\Theta > \theta\}) = {\tt pnorm}(\bar y - \theta), \quad \theta \in \RR, \]
and we see that there is, in fact, non-negligible support in the data for, say, the hypothesis ``$\Theta > 151$.''  These results agree exactly with the severity-based analysis presented in \citet{mayo.book.2018}.  The claim is that, for those $\theta$ whose value on red curve in Figure~\ref{fig:mean}(a) is relatively large, e.g., values near $\theta=151$ and perhaps up to $\theta=152$, the hypothesis ``$\Theta > \theta$'' garners non-negligible support from the data.  

\begin{figure}[t]
\begin{center}
\subfigure[$\bar y = 152$, ``reject $H_0$'']{\scalebox{0.6}{\includegraphics{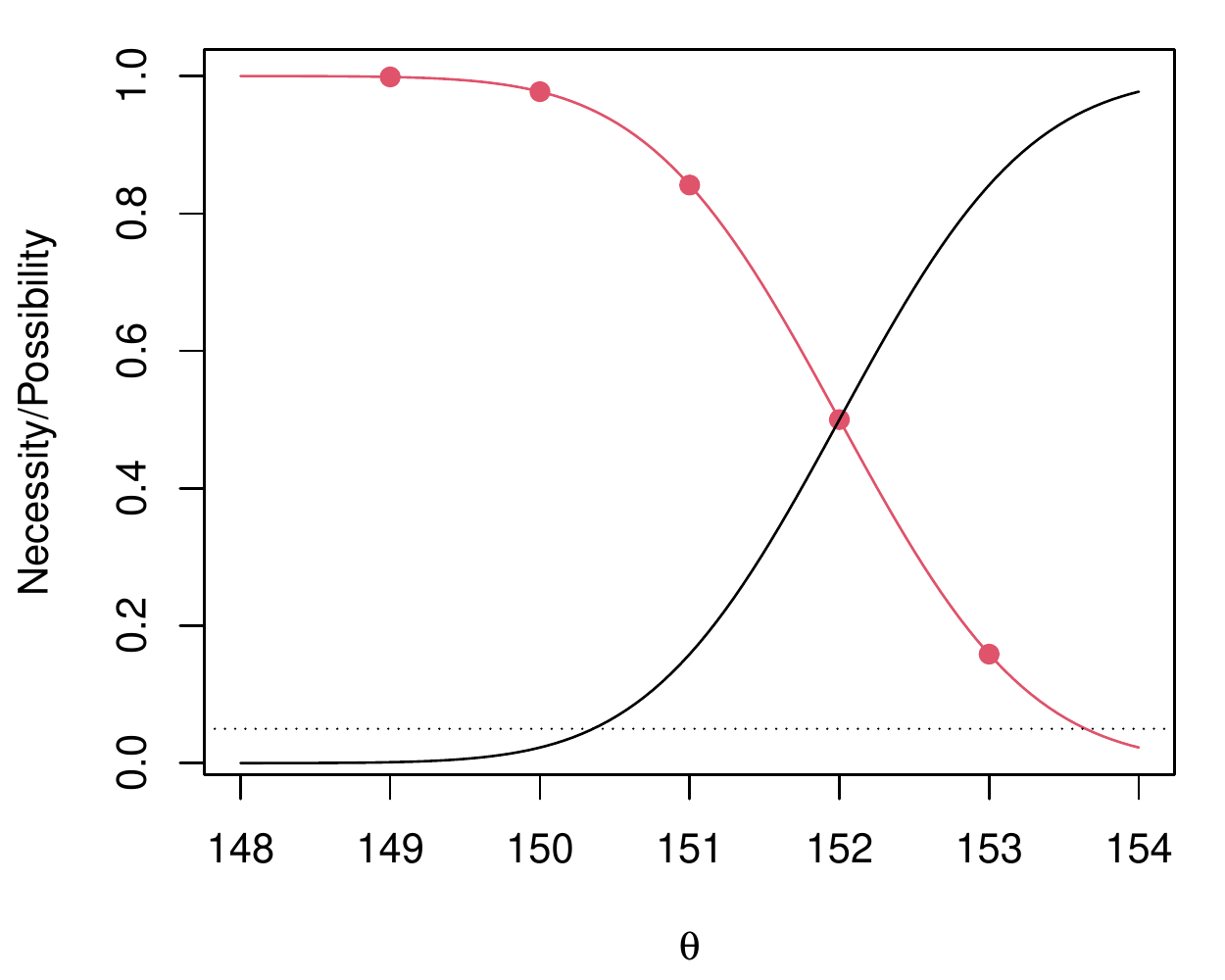}}}
\subfigure[$\bar y = 151$, ``do not reject $H_0$'']{\scalebox{0.6}{\includegraphics{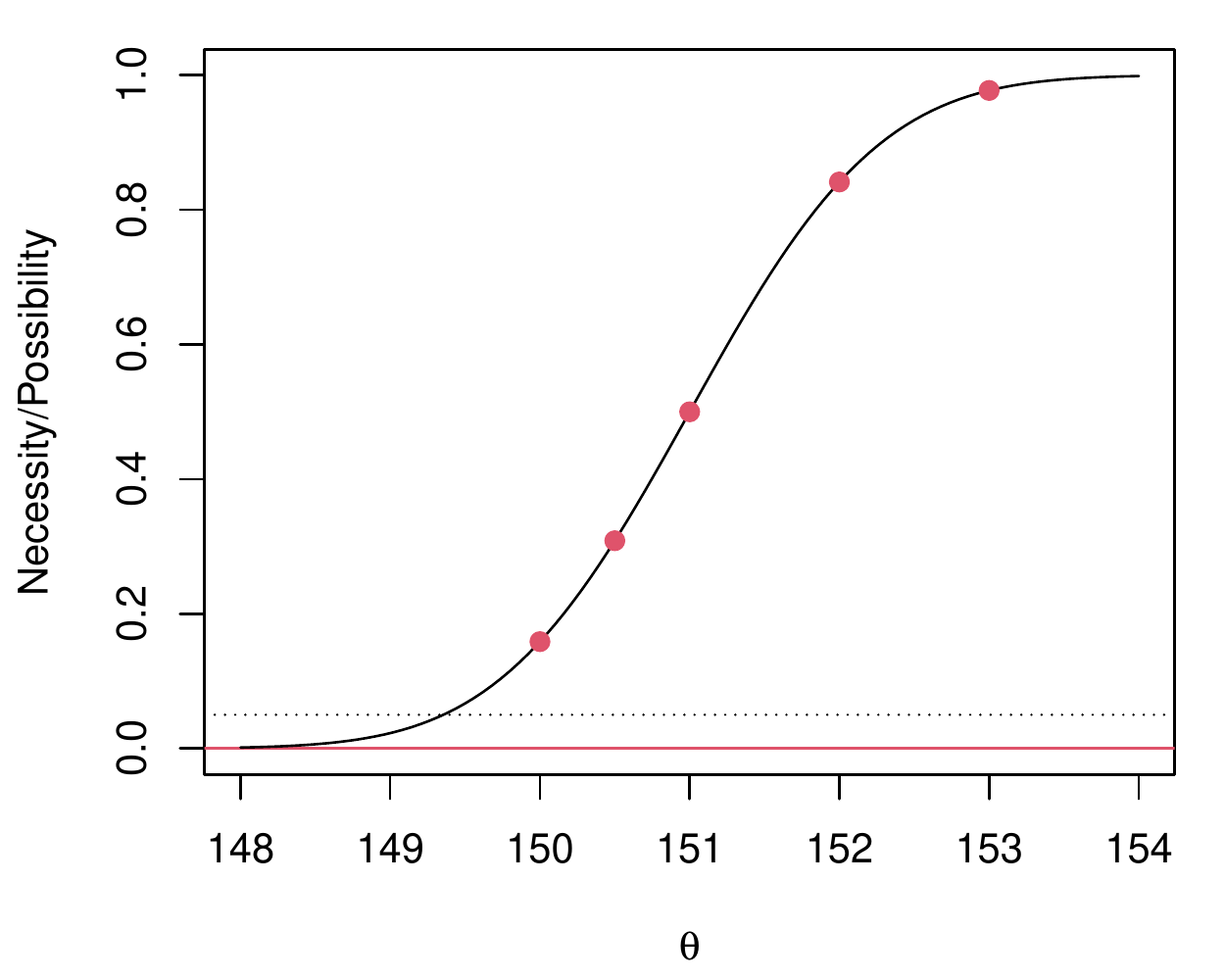}}}
\end{center}
\caption{Results for the normal mean example.  Panel~(a) shows a case where the null would be rejected, with $\theta \mapsto \uPi_y(\{\Theta \leq \theta\})$ in black and $\theta \mapsto \lPi_y(\{\Theta > \theta\}) \equiv 0$ in red; red dots correspond to the severity values in Table~3.1 of \citet{mayo.book.2018}. Panel~(b) shows a case where the null is not rejected, with $\theta \mapsto \uPi_y(\{\Theta \leq \theta\})$ in black and $\theta \mapsto \lPi_y(\{\Theta \leq \theta\}) \equiv 0$ in red; red dots correspond to the severity values in Table~3.3 of \citet{mayo.book.2018}.}
\label{fig:mean}
\end{figure}

Next, consider the case where $\bar y = 151$, which is too small to have grounds for rejecting $H_0$.  In such cases,  the goal would be to probe for potential support in hypotheses implied by the null.  Figure~\ref{fig:mean}(b) shows a plot of the possibility measure $\theta \mapsto \uPi_y(\{\Theta \leq \theta\})$, similar to that in Panel~(a), and necessity measure $\theta \mapsto \lPi_y(\{\Theta \leq \theta\}) \equiv 0$.  Also shown are the values of $\sev_y(\{\Theta \leq \theta\})$ for a few choice values of $\theta$ taken from Table~3.3 in \citet[][p.~145]{mayo.book.2018}, and there are two relevant points to note: 
\begin{itemize}
\item First, as expected in Case~2 situations like this one, Mayo's severity values drastically differ from the IM's necessity measure values.
\vspace{-2mm}
\item Second, thanks to the location parameter structure and symmetry of the normal model, the severity values actually agree with the IM's possibility measure values, which highlights the effects of this particular severity measure's greediness---it unexpectedly leads to simply treating the p-value as a measure of support!
\end{itemize} 
While the IM's trivially constant necessity measure is not a satisfactory solution to the probing problem, at least it does have the desirable error rate control property. Moreover, the IM's conclusion agrees with the understanding that even the most powerful test is incapable of offering support {\em for} the null hypothesis.  

Finally, we take a step back from the test-based focus, to approach the problem with a more holistic perspective.  This is intended to highlight the differences between the general IM framework and Mayo's severe testing approach that is centrally focused on a given test procedure.  The likelihood-based IM has possibility contour 
\[ \pi_y(\theta) = 1 - |2 \, {\tt pnorm}(|\bar y - \theta|) - 1|, \quad \theta \in \RR, \]
Figure~\ref{fig:mean.lik} shows plots of this contour function for two data sets: one with $\bar y=152$ and the other with $\bar y=151$.  These plots also show the possibility measure 
\[ \theta \mapsto \uPi_y(\{\Theta \leq \theta\}) = \sup_{\vartheta \leq \theta} \pi_y(\vartheta) = \begin{cases} \pi_y(\theta) & \text{if $\theta \leq \bar y$} \\ 1 & \text{if $\theta > \bar y$}, \end{cases} \]
which would be used to initially assess the null hypothesis.  As above, these two data sets correspond to Case~1 and Case~2, respectively, so we proceed differently with probing.  For the first data set, probing involves the necessity measure 
\[ \theta \mapsto \lPi_y(\{\Theta > \theta\}) = 1 - \uPi_y(\{\Theta \leq \theta\} = \begin{cases} 1-\pi_y(\theta) & \text{if $\theta \leq \bar y$} \\ 0 & \text{if $\theta > \bar y$}, \end{cases} \]
and, for the second data set, probing involves the necessity measure 
\[ \theta \mapsto \lPi_y(\{\Theta \leq \theta\}) = \begin{cases} 0 & \text{if $\theta \leq \bar y$} \\ 1-\pi_y(\theta) & \text{if $\theta > \bar y$}. \end{cases} \]
These two functions are plotted in red in Figure~\ref{fig:mean.lik}(a) and Figure~\ref{fig:mean.lik}(b), respectively.  In the former case, there is non-trivial support for the claim ``$\Theta > \theta$'' with $\theta$ up to roughly 151 or 151.5, but 0 support with $\theta$ more than 152.  In the latter case, there is non-trivial support for the claim ``$\Theta \leq \theta$'' with $\theta$ as small as 152 or perhaps 151.5.

\begin{figure}[t]
\begin{center}
\subfigure[$\bar y=152$, ``reject $H_0$'']{\scalebox{0.6}{\includegraphics{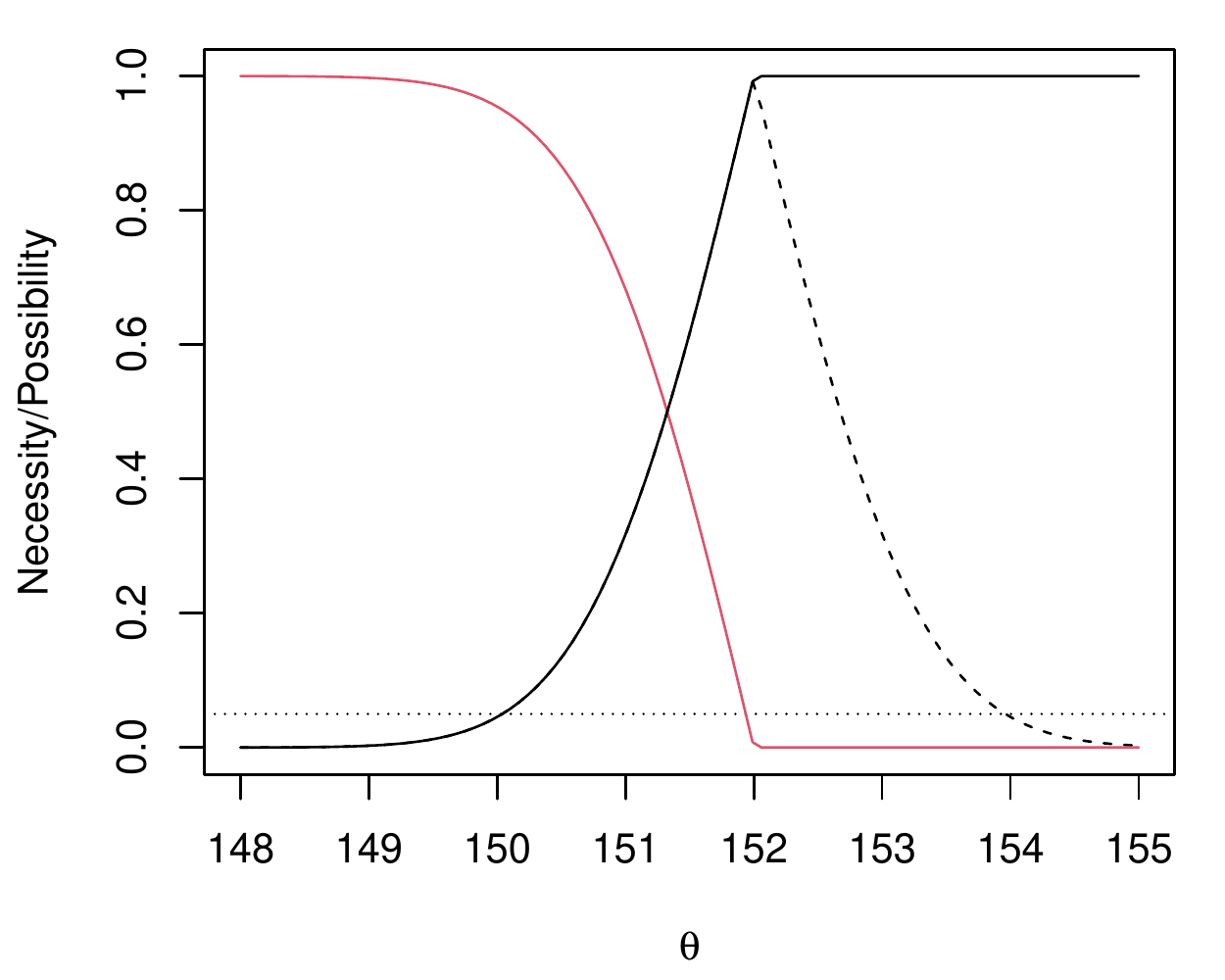}}}
\subfigure[$\bar y=151$, ``do not reject $H_0$'']{\scalebox{0.6}{\includegraphics{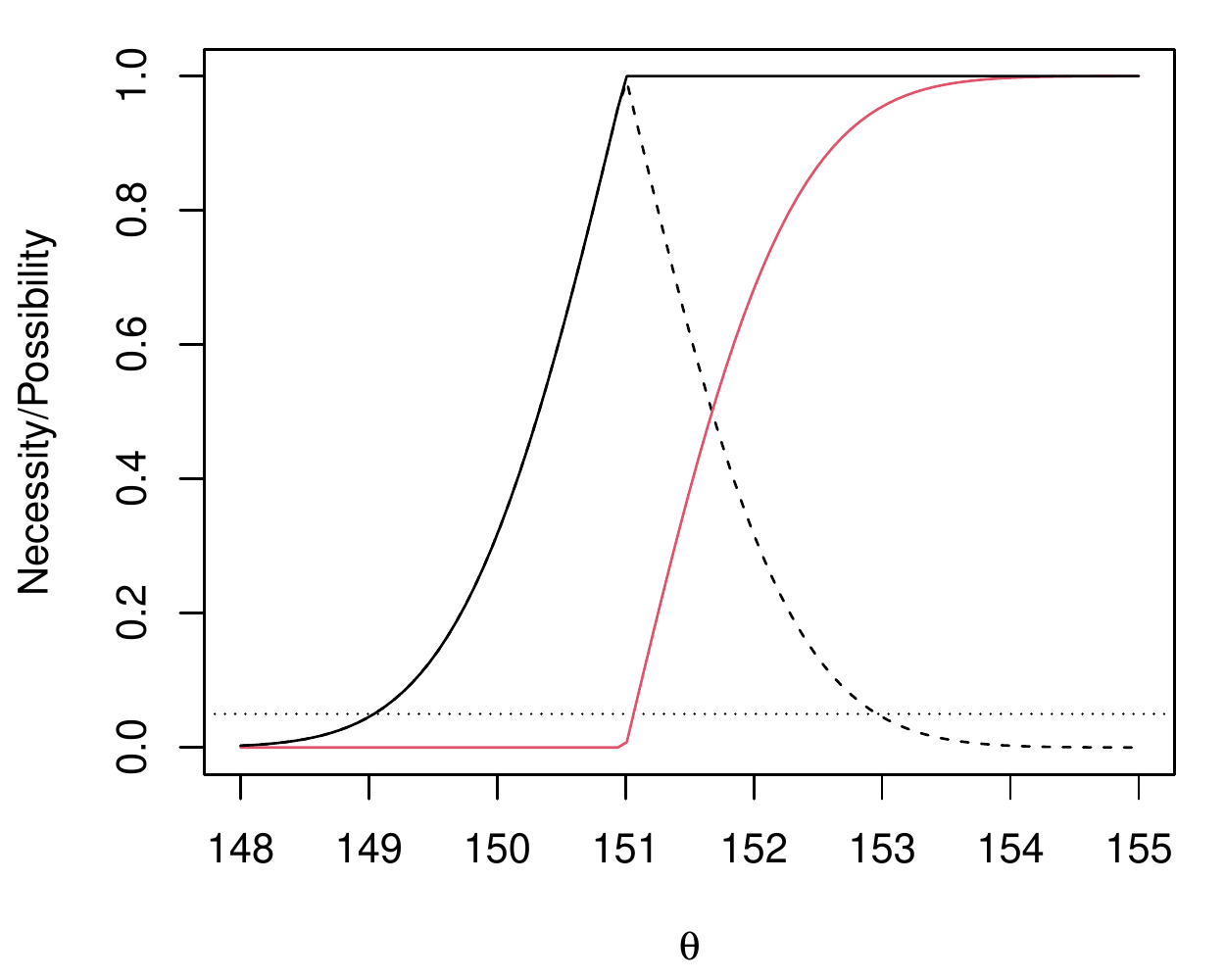}}}
\end{center}
\caption{Results of the holistic, likelihood-based IM for the normal mean problem.  Both panels show the possibility contour $\theta \mapsto \pi_y(\theta)$ and the possibility measure $\theta \mapsto \uPi_y(\{\Theta \leq \theta\})$ as dashed and solid black lines, respectively.  Panel~(a) corresponds to Case~1 so the necessity measure $\theta \mapsto \lPi_y(\{\Theta > \theta\})$ is shown in red.  Panel~(b) corresponds to Case~2 so the necessity measure $\theta \mapsto \lPi_y(\{\Theta \leq \theta\})$ is shown in red.}
\label{fig:mean.lik}
\end{figure}

\subsection{Binomial proportion}
\label{SS:bin}

Suppose an individual claims to possess psychic abilities. To test the validity of his claim, we set up the following experiment.  From a collection of five fixed symbols, a computer will generate one of these at random, and the claimed psychic will be asked to guess which of the five symbols the computer generated.  This will be repeated $n=20$ times and the result is a number $Y$ of correct guesses.  Then $Y$ has a binomial distribution with parameters $n=20$ and $\Theta \in [0,1]$ unknown.  Of course, given $Y=y$ the likelihood function for $\Theta$ is $L_y(\theta) \propto \theta^y(1-\theta)^{n-y}$, for $\theta \in [0,1]$. As a first step, we can carry out the likelihood-based construction in Section~\ref{SS:ims} to get an IM for $\Theta$.  Figure~\ref{fig:bin}(a) shows a plot of the possibility contour function based on an observed $y=8$ correct guesses out of $n=20$ trials.  The level set determined by the horizontal line at $\alpha=0.05$ determines a 95\% confidence interval for $\Theta$. 

To test the psychic's claim, the null hypothesis is $H_0: \Theta \leq \theta_0$, with $\theta_0 = 0.2$.  We can see from the possibility contour plot that $\uPi_y(H_0) < 0.05$, so we are inclined to reject the null.  But is there any support in the data for the psychic's claim?  For this, we probe hypotheses ``$\Theta > \theta$'' for $\theta > \theta_0$.  Figure~\ref{fig:bin}(b) plots the necessity $\theta \mapsto \lPi_y(\{\Theta > \theta\})$.  In this case, we find that hypotheses consisting of bold claims like ``$\Theta > \theta$'' for $\theta$ near 0.4 or even 0.5 are well supported by the data.


\begin{figure}[t]
\begin{center}
\subfigure[$\theta \mapsto \pi_y(\theta)$]{\scalebox{0.6}{\includegraphics{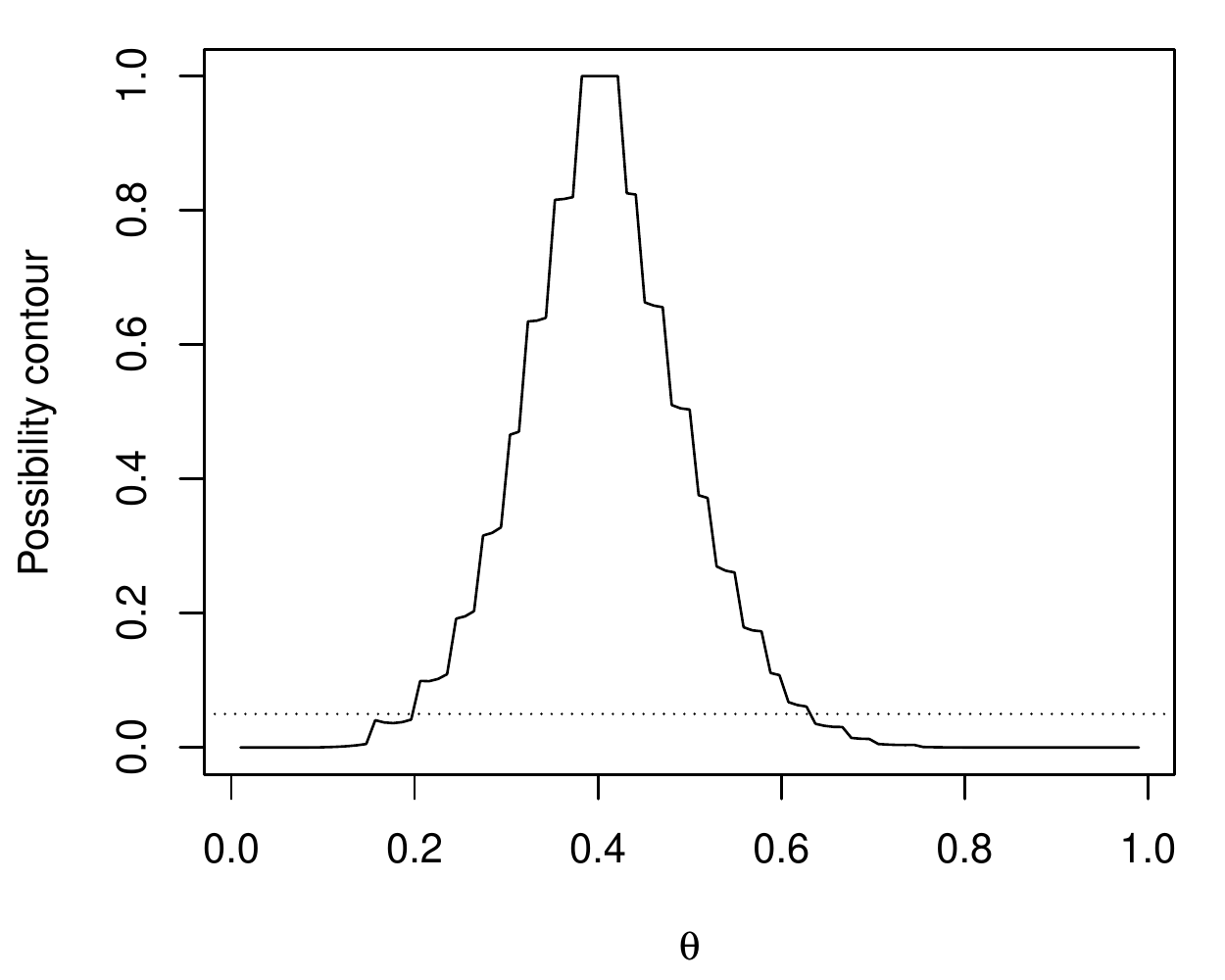}}}
\subfigure[$\theta \mapsto \lPi_y(\{\Theta > \theta\})$]{\scalebox{0.6}{\includegraphics{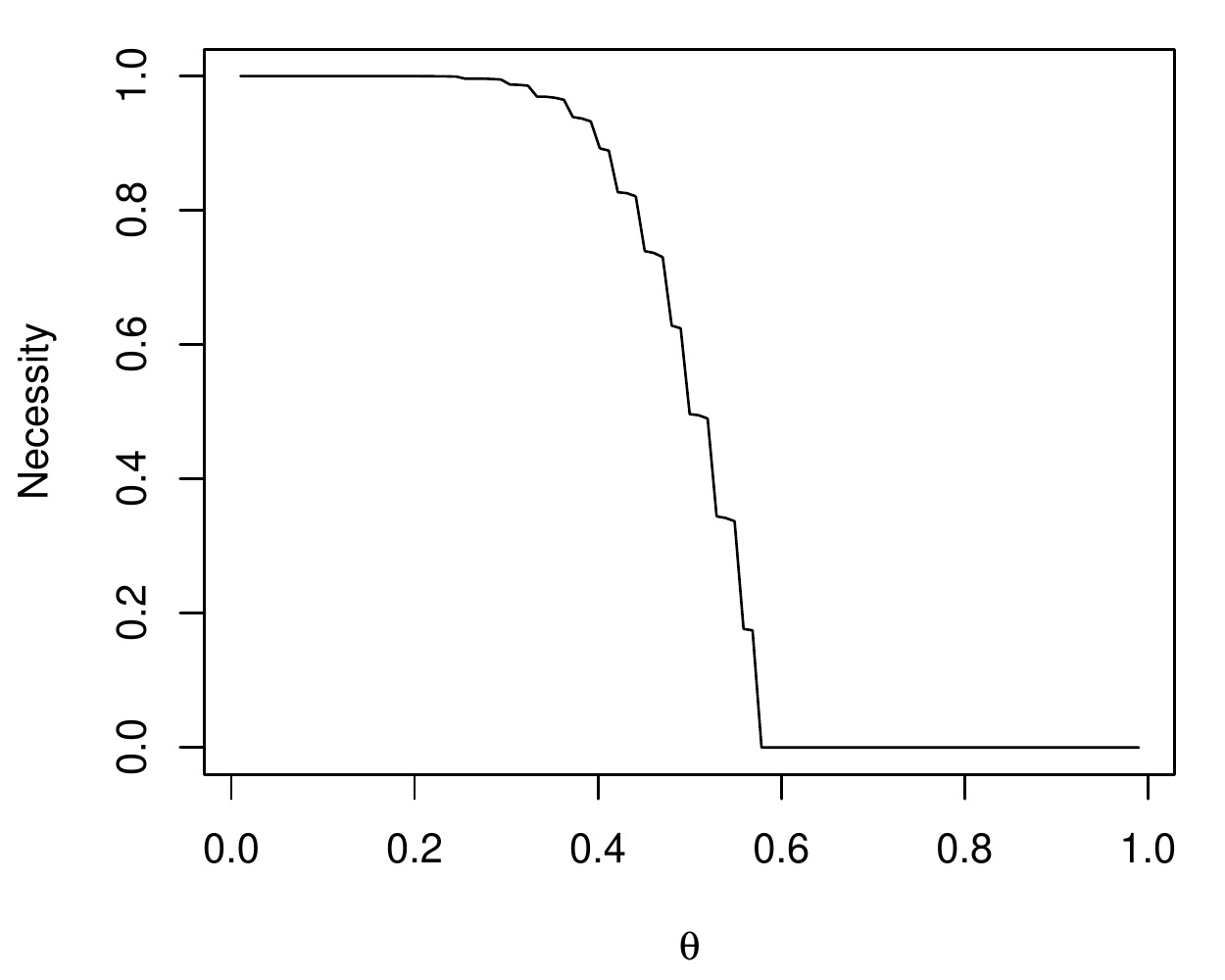}}}
\end{center}
\caption{IM results for the binomial example in Section~\ref{SS:bin}.}
\label{fig:bin}
\end{figure}

\subsection{Bivariate normal correlation}
\label{SS:corr}

Suppose that $Y$ consists of $n$ independent and identically distributed pairs $Y_i = (Y_{1,i}, Y_{2,i})$ having a bivariate normal distribution with zero means, unit variances, and correlation $\Theta \in [-1,1]$.  If there were a hypothesis $H_0: \Theta \leq \theta_0$, then an asymptotic pivot based on the maximum likelihood estimator, $\hat\theta$, could be constructed and the corresponding Wald test would look very similar to classical z-test like used in Section~\ref{SS:mean} above.  This bivariate normal correlation problem, however, corresponds to a so-called {\em curved exponential family}, so $\hat\theta$ is not a sufficient statistic and, consequently, some information/efficiency is lost in the aforementioned Wald test for finite $n$.  So we consider the more holistic approach and the likelihood-based IM construction from Section~\ref{SS:ims}.  The relative likelihood function has no closed-form expression, but it can be readily evaluated numerically.  Then the corresponding IM output, which requires probability calculations with respect to the bivariate normal model, can be found numerically using Monte Carlo.  


As an illustration of the ideas presented above, consider the law school admissions data analyzed in \citet{efron1982}, which consists of $n=15$ data pairs with $Y_1 = \text{LSAT scores}$ and $Y_2=\text{undergrad GPA}$.  For our analysis, we standardize these so that the mean zero--unit variance is appropriate.  Of course, this standardization has no effect on the correlation, which is our object of interest.  In this case, the sample correlation is 0.776; the maximum likelihood estimator, which has no closed-form expression, is $\hat\theta = 0.789$.  A plot of the plausibility contour $\pi_y$ for this data is shown in Figure~\ref{fig:cor}(a).  The horizontal line at $\alpha=0.05$ determines the 95\% plausibility interval for $\Theta$, which is an exact 95\% confidence interval.  Clearly, the data shows virtually no support for $\Theta=0$, but there is some marginal support for the hypothesis ``$\Theta > 0.5$.'' To probe this further, we plot the necessity measure $\theta \mapsto \lPi_y(\{\Theta > \theta\})$ in Figure~\ref{fig:cor}(b).  As expected from Panel~(a), the latter function is decreasing in $\theta$ and we clearly see no support for ``$\Theta > \theta$'' with $\theta \geq \hat\theta$.  But there is non-negligible support for ``$\Theta > \theta$'' with $\theta$ less than, say, 0.65--0.70. 

\begin{figure}[t]
\begin{center}
\subfigure[$\theta \mapsto \pi_y(\theta)$]{\scalebox{0.6}{\includegraphics{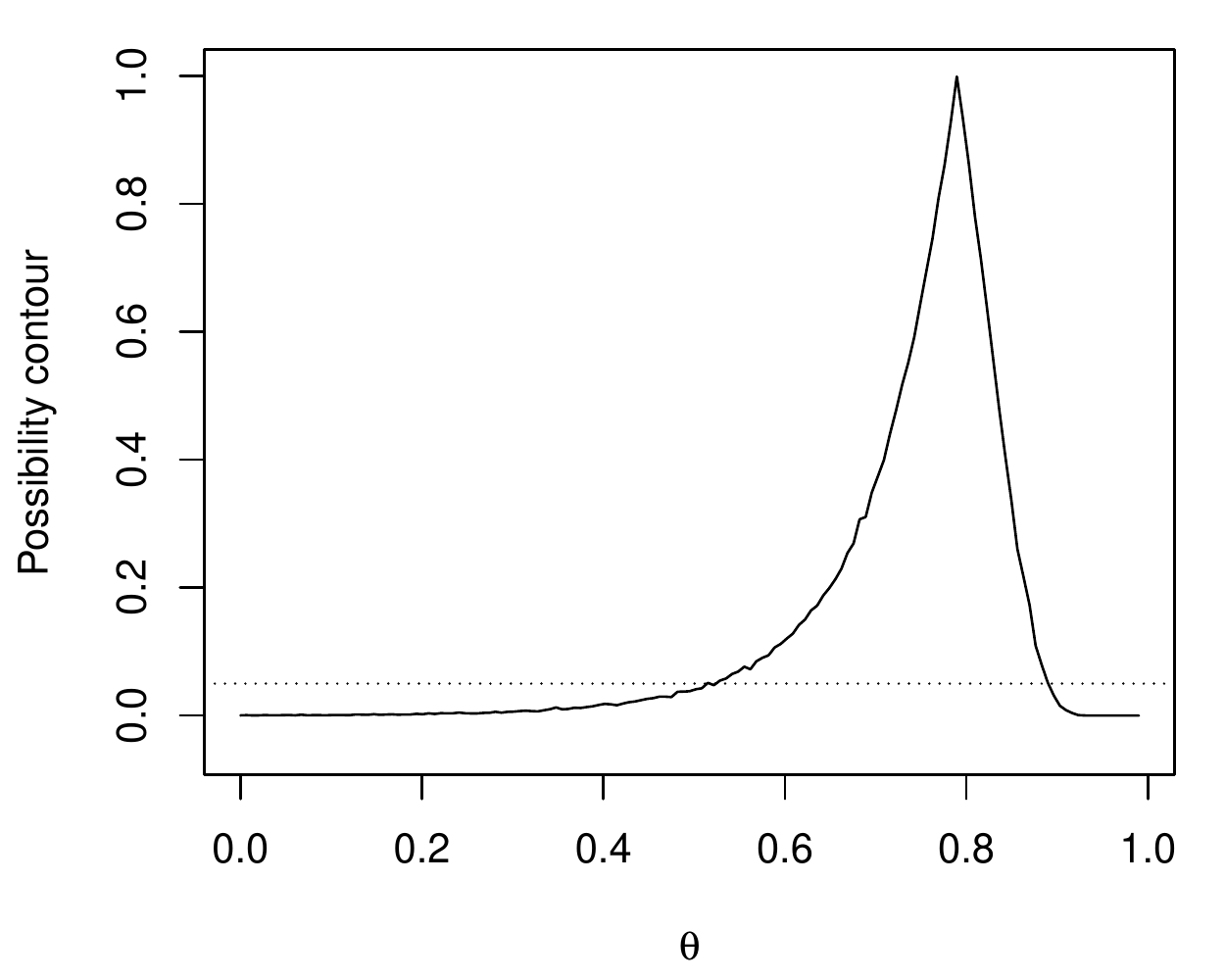}}}
\subfigure[$\theta \mapsto \lPi_y(H_\theta)$]{\scalebox{0.6}{\includegraphics{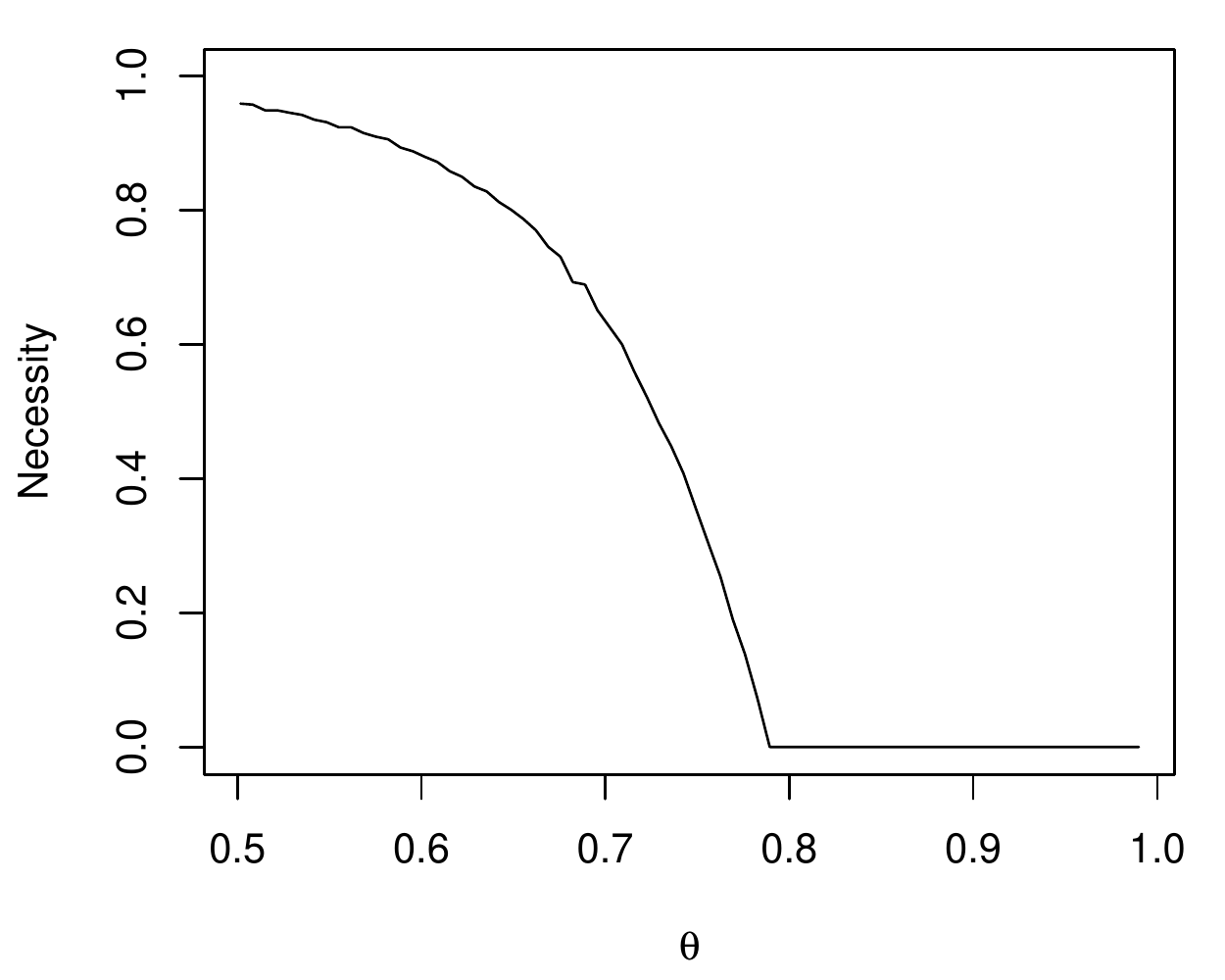}}}
\end{center}
\caption{Summary of the IM analysis of Efron's law school admissions data.}
\label{fig:cor}
\end{figure}

\subsection{Contingency tables}


Data $Y=(Y_{00}, Y_{01}, Y_{10}, Y_{11})$ represents the observed frequencies for each of the four combinations of two binary categorical variables $W$ and $X$, as shown in the $2 \times 2$ contingency table below.  $W$ and $X$ are the response and explanatory variables, respectively. 

\begin{table}[ht]
\centering
\begin{tabular}{c  c | c c | c}
 &  & \multicolumn{2}{c|}{$W$}  &    \\
& & 0 & 1 & Total \\
\hline
\multirow{2}{*}{{$X$}}& 0 & $y_{00}$ & $y_{01}$ & $y_{0 \cdot}$   \\
& 1 & $y_{10}$ & $y_{11}$ & $y_{1 \cdot}$  \\
\hline
& Total & $y_{\cdot 0}$ & $y_{\cdot 1}$ & $n$ 
\end{tabular}
\end{table}

The goal is to quantify uncertainty regarding the association between $W$ and $X$. In other words, to what extent does knowledge of the value of $X$ help us predict the value of $W$? Let $\Theta = (\Theta_0,\Theta_1)$, $\Theta \in [0,1]^2$, denote the conditional probabilities of $Y=1$ given $X=0$ and $X=1$, respectively; that is, 
\[ \Theta_x = \prob(W = 1 \mid X=x), \quad x \in \{0,1\}. \]
The association between $W$ and $X$ can be stated in terms of the difference $\Theta_0 - \Theta_1$. A difference of zero implies no association, and the bigger the difference, positive or negative, the stronger the association. 

To construct an IM for $\Theta_0 - \Theta_1$, we can leverage the marginalization properties of the IM. Our approach involves first constructing an IM for $\Theta$, and then mapping it to a marginal IM for $\Phi = f(\Theta) = \Theta_0 - \Theta_1$. It is important to note that the method for collecting data is fundamental to the specification of the likelihood function $L_y(\theta)$ and, consequently, to the holistic IM construction from Section~\ref{SS:ims}.  
Here, we will consider the scenario where the row totals are fixed. This means that random samples of size $y_{0\cdot}$ and $y_{1\cdot}$ observations are drawn from the populations corresponding to $X=0$ and $X=1$, respectively, and then each observation is classified as either $W=0$ or $W=1$. This experiment produces two independent sequences of Bernoulli trials, therefore, a product of binomial likelihoods: 
\[ L_y(\theta) \propto \theta_{0}^{y_{01}}(1-\theta_0)^{y_{0\cdot}-y_{01}} \times \theta_{1}^{y_{11}}(1-\theta_1)^{y_{1\cdot}-y_{11}}, \quad \theta = (\theta_0, \theta_1) \in [0,1]^2.\]


Consider a hypothetical clinical trial in which $n=50$ participants are randomly and equally divided into two groups. One group receives a drug, the other a placebo. After a year, the participants undergo an evaluation to determine whether a specific aspect of their health has improved. Table~\ref{t:clinical} shows the data and Figure~\ref{fig:2x2}(a) shows a contour plot of the possibility contour for $\Theta$. The dotted lines correspond to $\hat \theta_y$, the maximum likelihood estimator. Recall that a marginal possibility contour for $\Phi$ can be obtained from the possibility contour for $\Theta$ through \eqref{eq:featureIMcontour}. Figure~\ref{fig:2x2}(b) shows the corresponding necessity and possibility measures for hypotheses $H_{\phi} = (-1,\phi]$ and $H^c_{\phi} = (\phi,1)$, respectively.
Note that $\uPi_y^f(H_0)$ is small, which would make us inclined to reject the hypothesis ``$\Phi\leq 0$.'' Non-negligible necessity measures for $H^c_{\phi}$ with $\phi < 0.1$ can be observed. 
The public-health importance of raw differences like this may be hard to grasp in some applications. In such cases, data analysts often prefer to use the concept of {\em relative risk} $\Phi = g(\Theta) =  \Theta_0/\Theta_1$. Once again, it is straightforward to obtain a possibility contour for $\Phi$ from that for $\Theta$. Figure~\ref{fig:2x2}(c) shows this possibility contour. \citet{agresti2007introduction} points out the difficulties in deriving confidence intervals for the relative risk because of the highly skewed distribution of the plug-in estimator. An advantage of our marginal IM is that a 95\% confidence interval for $\Phi$ is readily available from the possibility contour in Figure~\ref{fig:2x2}(c), being the level set determined by the horizontal line at $\alpha=0.05$. As expected, this interval does not contain $\phi=1$, so the null hypothesis of no association can be rejected.  Figure~\ref{fig:2x2}(d) shows the marginal IM necessity measures for hypotheses $H_{\phi} = (\phi,\infty)$. Evidently there is reasonably strong support for the alternative sub-hypothesis that the risk of disease is at least 20\% higher in the placebo group.


\begin{table}[t]
\centering
\begin{tabular}{c | c c | c}
&\multicolumn{2}{c|}{Disease}&\\
\hline
Group & No & Yes & Total \\
\hline
 Placebo & 11 & 14 & 25   \\
 Drug & 17 & 8 & 25  \\
\hline
\end{tabular}
\caption{Hypothetical clinical trial data.}
\label{t:clinical}
\end{table}

\begin{figure}[t]
\begin{center}
\subfigure[Level sets of $\theta \mapsto \pi_y(\theta)$]{\scalebox{0.46}{\includegraphics{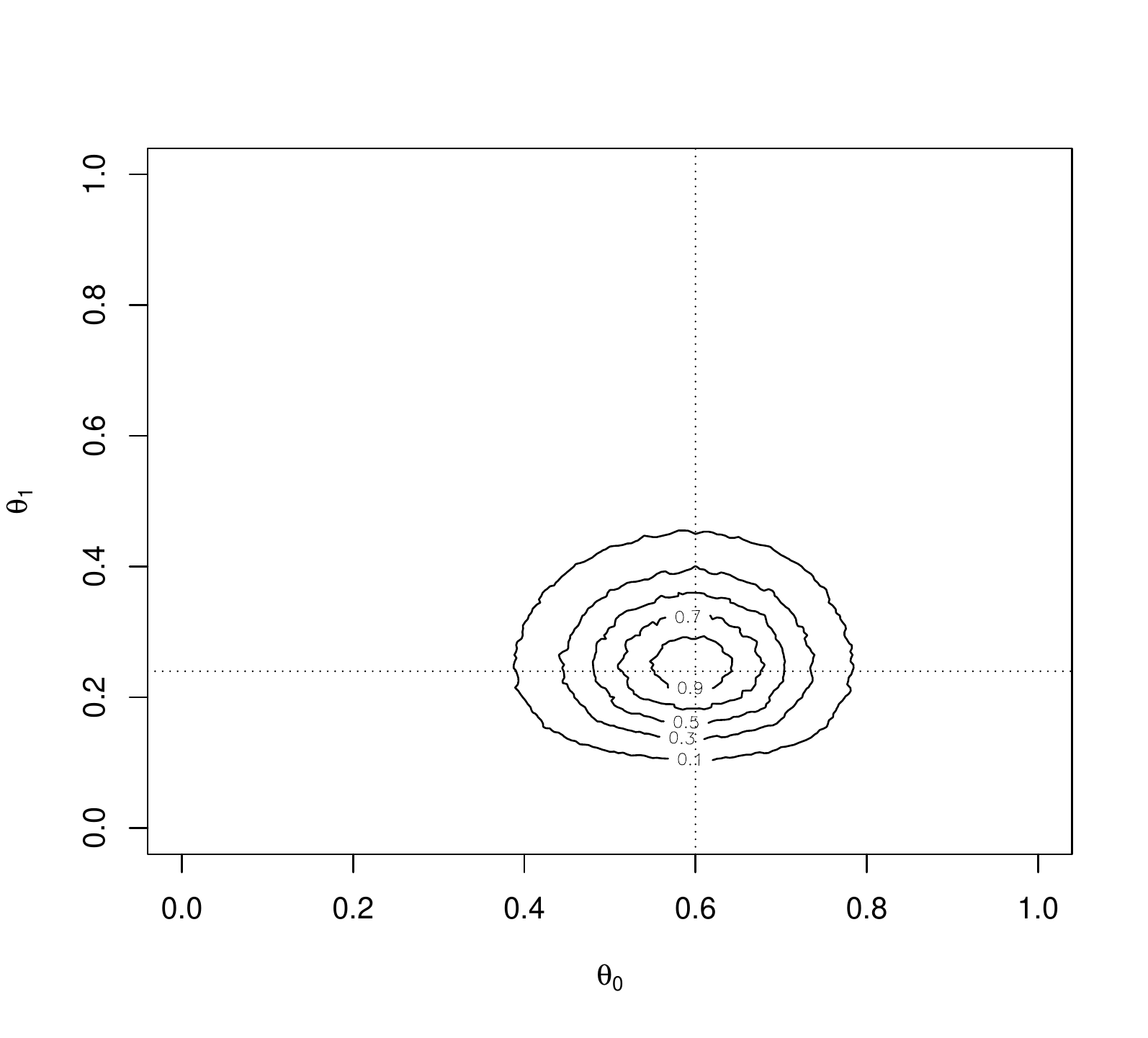}}}
\subfigure[$\phi \mapsto (\lPi_y^{f}(H^c_{\phi}),\uPi_y^{f}(H_{\phi}))$]{\scalebox{0.46}{\includegraphics{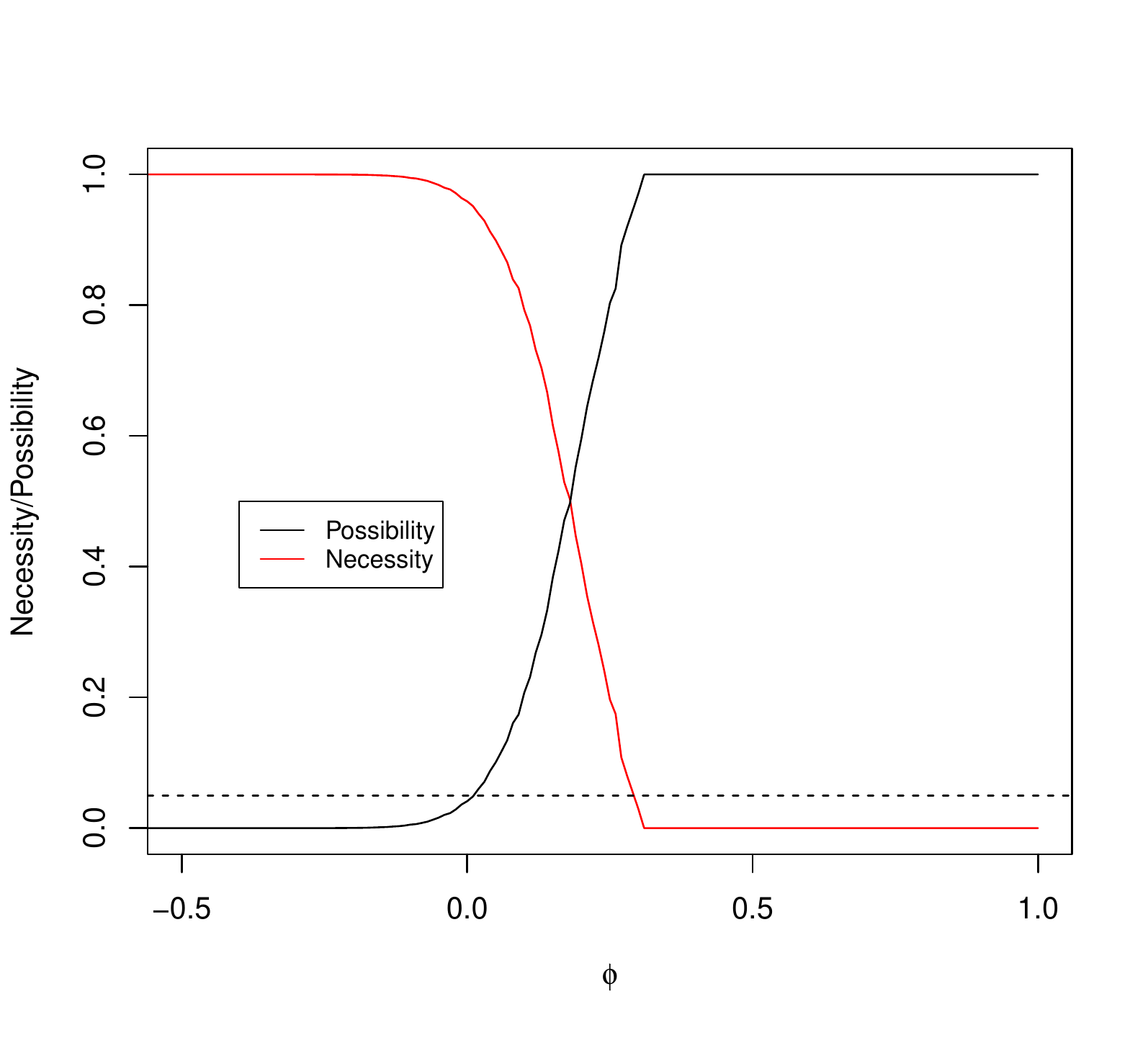}}}
\subfigure[$\phi \mapsto \pi_y^g(\phi)$ ]{\scalebox{0.46}{\includegraphics{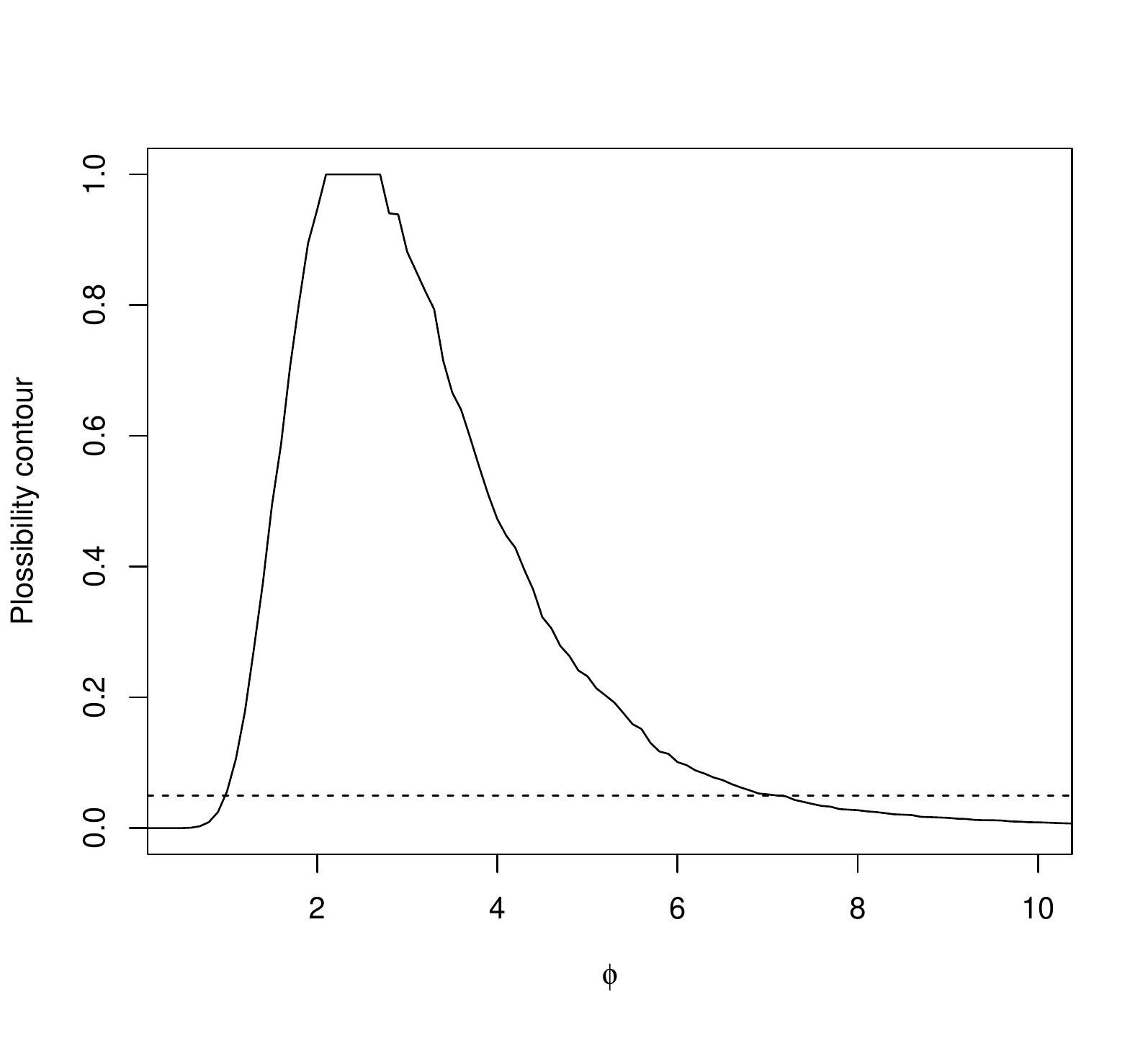}}}
\subfigure[$\phi \mapsto \lPi_y^g(H_{\phi})$]{\scalebox{0.46}{\includegraphics{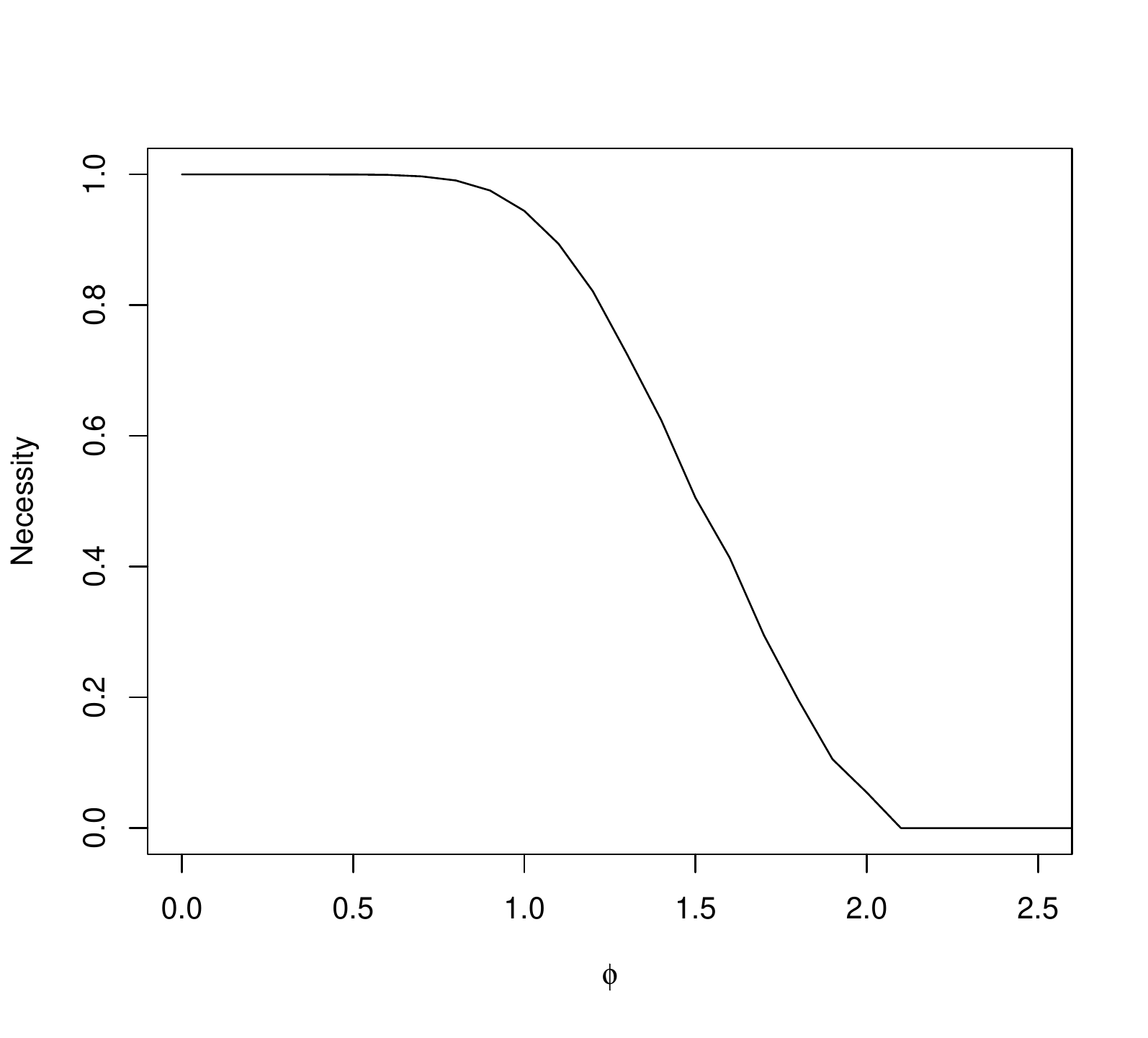}}}
\end{center}
\caption{Results of the valid IM applied to the hypothetical data in Table~\ref{t:clinical}. }
\label{fig:2x2}
\end{figure}

\section{Conclusion}
\label{S:discuss}

Here we showed that there is more to the IM framework than what has been presented in the existing literature. Specifically, the validity property, together with its inherent imprecision implies not only performance, but also probativeness assurances. These insights position IMs as a compelling solution to the long-standing Bayesian versus frequentist two-theory problem, which will undoubtedly benefit the statistical community. They are also of special interest to the belief function/possibility theory community, as it showcases the fundamental importance of their specific brand of imprecision.  

We also compared IMs and Mayo's severe testing framework, and found important points of convergence and divergence.  Specifically, in what we referred to here as ``Case~1'' situations, where the null hypothesis being tested is rejected, and the goal is to probe for support in hypotheses that imply the alternative, Mayo's proposal agrees with (a particular version of) our proposal.  In the ``Case~2'' situations, however, where the null hypothesis is not rejected and the goal is to probe for support in hypotheses implied by the null, the two proposals are very different.  We argue that the proposed IM approach, and its solution to the probing problem, is more appropriate in these Case~2 scenarios by appealing to both common sense and the IM's inherent strong reliability properties.  More generally, Mayo's idea to treat probing as an embellishment on the existing statistical hypothesis testing framework actually hinders ones ability to probe reliably, while what we called the ``holistic'' IM solution is both flexible and provably reliable.  A number of illustrations are presented in Section~\ref{S:examples}; application of these methods in cases beyond these simple, low-dimensional problems will be reported elsewhere.  In particular, in light of recent developments in \citet{cella.martin.imrisk} we are hopeful that a notion of probativeness to model-agnostic statistical learning problems is within reach.

\section*{Acknowledgments}

The authors thank two anonymous reviewers for valuable feedback on a previous version of this manuscript.  This research is partially supported by the U.S.~National Science Foundation, SES--2051225.

\appendix

\section{Hypothesis testing details}
\label{app:test.im}

\subsection{Test-based IM construction}
\label{app:test.im1}

As an alternative to the likelihood-based IM construction in Section~\ref{SS:ims}, \citet{imchar} first showed how to construct an IM driven by a given statistical procedure, i.e., a hypothesis test of a confidence region.  Since that statistical procedure is often tailored to a specific task or question, e.g., testing a particular (form of) hypothesis, this construction would tend to be ``less holistic'' than the likelihood-based construction mentioned above.  It can happen, however, that the two constructions agree, as we will demonstrate below.  To us, the ``more holistic'' likelihood-driven construction is preferred, but the procedure-driven strategy has its own advantages. 

Here we focus the procedure-based construction on cases where a family of hypothesis testing problems are given.  Start with a class of hypotheses $\{H_\theta: \theta \in \TT\}$ about $\Theta$ indexed by the parameter space $\TT$.  These could be singleton/point-null hypotheses, $H_\theta = \{\theta\}$, half-line hypotheses, $H_\theta = (-\infty, \theta]$, in the case of scalar $\Theta$, or other things.  Next, consider a collection $\{\delta_\alpha^\theta: \alpha \in [0,1], \theta \in \TT\}$ of decision rules, where $\delta_\alpha^\theta: \YY \to \{0,1\}$, with the interpretation that $\delta_\alpha^\theta(y) = 1$ means reject $H_\theta$ and $\delta_\alpha^\theta(y)=0$ means do not reject.  For example, in a simple scalar location parameter setting, where $H_\theta = (-\infty, \theta]$, then the testing rule might take the form $\delta_\alpha^\theta(y) = 1(y - \theta > c_\alpha)$, where $c_\alpha$ is a specified threshold and $1(\cdot)$ is the indicator function.  The index $\alpha$ controls the size or Type~I error probability of the test:
\begin{equation}
\label{eq:size}
\sup_{\Theta \in H_\theta} \prob_\Theta\{ \delta_\alpha^\theta(Y) = 1 \} \leq \alpha, \quad \text{for all $\alpha \in [0,1]$, $\theta \in \TT$}. 
\end{equation}

The reader might be asking him/herself why a {\em family} of tests indexed by $\theta \in \TT$ would be needed.  Keep in mind that the IM returns a full-blown imprecise probability defined over $\TT$, so it would be unrealistic to expect that anything meaningful could be obtained based on a test of, say, a single hypothesis.  Often the problem structure suggests a form of hypothesis, e.g., $H: \Theta \leq \theta_0$, and that the same testing procedure would have been used if $\theta_0$ was changed to $\theta_0 + \eta$.  So even if one has just one specific hypothesis/test procedure in mind, often that one belongs to a family like described above, so specifying a family imposes no additional burden on the data analyst.  

In any case, from here, define the function 
\begin{equation}
\label{eq:contour.test}
\pi_y(\theta) = \inf\{\alpha \in [0,1]: \delta_\alpha^\theta(y) = 1\}, \quad \theta \in \TT. 
\end{equation}
This is just the p-value function corresponding to the collection of tests \citep[][Eq.~3.11]{lehmann.romano.2005}.  Under certain conditions on the test (see below), $\theta \mapsto \pi_y(\theta)$ is a genuine possibility contour function on $\TT$, i.e., $\sup_\theta \pi_\theta(y)=1$ for each $y$.  In that case, the IM's possibility and necessity measures are defined via optimization exactly like in \eqref{eq:IMposs} and will enjoy all the same properties, e.g., \eqref{eq:order}.  

What conditions are required of the collection of tests to ensure that the function defined in \eqref{eq:contour.test} is a possibility contour?  Basically, the collection of tests needs to satisfy a certain ``nestedness'' condition.  The concept itself is pretty simple---for each data set $y$, there is a hypothesis $H_\theta$ that cannot be rejected at any level $\alpha$---but a precise mathematical statement is complicated.  In the simplest case, suppose that for each $y$, there exists $\theta$ such that $\delta_\alpha^\theta(y)=0$ for all $\alpha \in [0,1]$, i.e., that there is a hypothesis $H_\theta$ that cannot be rejected based on data $y$.  For example, if $Y$ denotes a sample of size $n$ from a normal distribution with mean $\Theta$ and known variance $\sigma^2$, and if the hypotheses $H_\theta = \{\theta\}$ are singletons, then the usual z-test cannot reject $H_\theta$ with $\theta=\bar y$ at any significance level $\alpha$.  In general, suppose that for each $y$, there exists a net $\alpha \mapsto \theta_y(\alpha)$, for $\alpha \in [0,1]$, such that $\delta_\alpha^{\theta_y(\alpha)}(y) = 0$ for all $\alpha$ sufficiently close to 1.  In the previous normal illustration, for any data $y$, the half-line hypothesis $H_\theta = (-\infty,\theta]$, with 
\[ \theta = \theta_y(\alpha) = \bar y - c \, z_{1-\alpha} \, \sigma \, n^{-1/2}, \quad \text{any $c \in (0,1]$}, \]
would not be rejected for any $\alpha \in [0,1]$.

\subsection{When do the two IM constructions agree?}

The two IM constructions above will agree when the procedure-driven construction is based on the likelihood ratio test for the class $H_\theta = \{\theta\}$ of singleton hypotheses.  In this case, the test procedure could be described by the rule
\[ \delta_\alpha^\theta(y) = 1\{ R(y,\theta) \leq c_\alpha(\theta)\}, \quad y \in \YY, \quad \theta \in \TT, \quad \alpha \in [0,1], \]
where $1(\cdot)$ is the indicator function and $c_\alpha(\theta)$ is chosen to ensure that \eqref{eq:size} holds.  As is well known, thanks to the definition of $c_\alpha(\theta)$ through the sampling distribution of $R(Y,\theta)$ under $\prob_\theta$, it follows that the testing rule can be equivalently expressed as 
\[ \delta_\alpha^\theta(y) = 1\{ \pi_y(\theta) \leq \alpha \}, \quad y \in \YY, \quad \theta \in \TT, \quad \alpha \in [0,1], \]
where $\pi_y(\theta)$ is the p-value/contour in \eqref{eq:IMcontour}.  It is now clear that 
\[ \inf\{ \alpha: \delta_\alpha^\theta(y) = 1\} = \inf\{\alpha: \pi_y(\theta) \leq \alpha \} = \pi_y(\theta), \]
so the contour function defined in \eqref{eq:contour.test} agrees with that in \eqref{eq:IMcontour}.

\subsection{Simplification in Mayo's context}
\label{app:mayo}

To our knowledge, Mayo's developments focus on a special class of problems involving a scalar location parameter $\Theta$ and one-sided null hypotheses like $H_0: \Theta \leq \theta_0$.  There are lots of problems that fit this setting, at least asymptotically, so she has grounds to make these cases her focus.  The relevant structure below also holds more generally---but still scalar parameter---where the model admits a {\em monotone likelihood ratio} property \citep[e.g.,][]{casella.berger.book, karlin.rubin.1956}.  

Recall the setup described in Section~\ref{SS:mayo.back}, where the test procedure rejects the null hypothesis $H_0: \Theta \leq \theta_0$ based on data $y$ if and only if $S(y, \theta_0)$ exceeds some specified threshold $s_\alpha$, indexed by $\alpha \in [0,1]$.  What the location structure implies is that the test statistic can be written as $S(y,\theta_0) = S(y, 0) - \theta_0$.  The point is that there is nothing special about the value $\theta_0$, i.e., the test has exactly the same form if $\theta_0$ is replaced by, say, $\theta_0 + \eta$.  So the single test $\delta_\alpha^{\theta_0}(y) = 1\{S(y,0) > \theta_0 + s_\alpha\}$ determines a family of tests, 
\[ \delta_\alpha^\theta(y) = 1\{S(y,0) > \theta_0 + s_\alpha\}, \quad \theta \in \RR, \quad \alpha \in [0,1], \]
one for each hypothesis $H_\theta = (-\infty,\theta]$, for $\theta \in \RR$.  Then we follow the general strategy in Appendix~\ref{app:test.im1} above---in particular, the definition in \eqref{eq:contour.test}---to get that 
\[ \pi_y(\theta) = \inf\{\alpha \in [0,1]: \delta_\alpha^\theta(y) = 1\} = \pval_y^\leq(\theta), \quad \theta \in \RR. \]
This proves the claim in \eqref{eq:mayo.im.special}.


\bibliographystyle{apalike}
\bibliography{mybib}

\end{document}